\documentclass[hidelinks,onefignum,onetabnum, final]{siamart251216}

\ifpdf
\hypersetup{
  pdftitle={SDP for ternary quadratic problems},
  pdfauthor={F. de Meijer, V. Piccialli, R. Sotirov, A.M. Sudoso}
}
\fi

\usepackage[english]{babel}
\usepackage[short]{optidef}
\usepackage{amsmath}
\usepackage{amssymb}
\usepackage{bbm}
\usepackage{graphicx}
\usepackage{booktabs}
\usepackage{tabularx}
\usepackage{comment}
\usepackage{adjustbox}
\usepackage{longtable}
\usepackage{enumitem}
\usepackage{algorithm}
\usepackage{algorithmic}
\usepackage{xcolor}
\usepackage{multirow}
\usepackage{nicematrix}
\usepackage{tikz}
\usepackage{cleveref}
\usepackage{todonotes}
\usepackage{subcaption}

\newcommand{\todoinlinefrank}[1]{\todo[inline,author=Frank,color=black!20]{#1}}

\DeclareMathOperator*{\argmin}{arg\,min}

\newcommand{\st}{\textrm{s.t.}}
\newcommand{\sn}{{\mathcal S}^n}

\newcommand{\Diag}{\mathrm{Diag}}

\newcommand{\Col}{\mathrm{Col}}
\newcommand{\diag}{\mathrm{diag}}
\newcommand{\rank}{\mathrm{rank}}
\newcommand{\tr}{\mathrm{tr}}
\newcommand{\conv}{\mathrm{conv}}
\newcommand{\R}{\mathbb{R}}
\newcommand{\N}{\mathbb{N}}
\newcommand{\Z}{\mathbb{Z}}
\newcommand{\Q}{\mathbb{Q}}

\definecolor{ForestGreen}{RGB}{25, 85, 55}
\allowdisplaybreaks

\newtheorem{remark}{Remark}

\NiceMatrixOptions
{
    custom-line ={command= H, tikz= dashed, width= 1mm}, 
    custom-line = {letter= I, tikz= dashed, width= 1mm}, 
}

\title{Beyond binarity: Semidefinite programming for ternary quadratic problems}

\author{
Frank de Meijer\thanks{Delft Institute of Applied Mathematics, Delft University of Technology, The Netherlands. \texttt{f.j.j.demeijer@tudelft.nl}}
\and 
Veronica Piccialli\thanks{Department of Computer, Control and Management Engineering ``Antonio Ruberti", Sapienza University of Rome, Italy.   \texttt{veronica.piccialli@uniroma1.it}}
\and
Renata Sotirov\thanks{CentER, Department of Econometrics and OR, Tilburg University, The Netherlands. \texttt{r.sotirov@uvt.nl}}
\and 
Antonio M. Sudoso\thanks{Department of Mathematics ``Tullio Levi-Civita”, University of Padua, Italy. \texttt{antoniomaria.sudoso@unipd.it}}
}

\date{\today}

\begin{document}
\maketitle 

\begin{abstract}
   We study the ternary quadratic problem (TQP), a quadratic optimization problem with linear constraints where the variables take values in $\{0,\pm 1\}$. While semidefinite programming (SDP) techniques are well established for $\{0,1\}$- and $\{\pm 1\}$-valued quadratic problems, no dedicated integer semidefinite programming framework exists for the ternary case. In this paper, we introduce a ternary SDP formulation for the TQP that forms the basis of an exact solution approach.

   We derive new theoretical insights in rank-one ternary positive semidefinite matrices, which lead to a basic SDP relaxation that is further strengthened by valid triangle, RLT, split and $k$-gonal inequalities. These are embedded in a tailored branch-and-bound algorithm that iteratively solves strengthened SDPs, separates violated inequalities, applies a ternary branching strategy and computes high-quality feasible solutions. We test our algorithm on TQP variations motivated by practice, including unconstrained, linearly constrained and quadratic ratio problems. Computational results on these instances demonstrate the effectiveness of the proposed algorithm.
\end{abstract}



\section{Introduction}
Integer semidefinite programming (ISDP) may be seen as a generalization of integer linear programming, 
where the vector of variables is replaced by a positive semidefinite (PSD) matrix variable whose elements are constrained to be integers.
Recent results show that integer semidefinite programs allow for more compact and substantially different formulations
for several problem classes compared to traditional modeling approaches, and also enable the development of new algorithmic techniques. 
These programs do not include $\mathcal{NP}$-hard rank constraints, but utilize an alternative notion of exactness in terms of integrality of  PSD matrices.
Although it is straightforward to embed a linear integer program (IP)  into the space of diagonal integer PSD matrices,
it is  generally not clear how to construct an ISDP program for a  quadratic IP.
However, it is  well known that quadratic optimization problems with variables in $\{\pm 1\}$, e.g., the  max-cut problem, 
can be formulated as  ISDP  problems, see e.g.,~\cite{DezaLaurent1997,GoemansWilliamson}. 
It is also known how to reformulate  binary quadratically constrained quadratic programs and
binary quadratic matrix programs as binary semidefinite programs \cite{DeMeijerSotirov23}.
In this work, we introduce an ISDP formulation for  the ternary quadratic problem (TQP), that is an optimization problem with quadratic objective and linear constraints, where the decision variables take values in $\{0, \pm 1\}$. 
To the best of our knowledge, there are no up to date results on ternary semidefinite programs.

The TQP belongs to the class of  nonlinear IPs,  which have received increased attention in recent years. 
Ternary optimization problems are frequently  studied in the literature as special cases of  quadratic integer programs.
In 2010, Letchford~\cite{10.1007/978-3-642-13036-6_20} studied  integer quadratic quasi-polyhedra, i.e., 
polyhedra with a countably infinite number of facets.
Such polyhedra arise as feasible sets of unconstrained integer quadratic programming problems.
Letchford’s work was extended to convex sets for nonconvex mixed-integer quadratic programming in~\cite{Burer2014UnboundedCS}.
{\color{black}In~\cite{Buchheim2012SemidefiniteRF,Buchheim2018SDPbasedBF} were proposed} semidefinite programming (SDP) relaxations for nonconvex mixed-integer quadratic optimization problems, including an SDP relaxation for the TQP.
These relaxations are derived from exact problem reformulations that involve a nonconvex rank-one constraint.

Approaches for solving optimization problems with ternary variables also appear in the broader literature on general quadratic integer programming, where the TQP is often considered in computational experiments. 
In the following, we provide a list of these works.
In 2012, {\color{black}the authors of}~\cite{christoph_2011} present a branch-and-bound (B\&B)  algorithm for solving convex quadratic unconstrained IPs for instances of up to 120 variables. 
A B\&B algorithm for nonconvex quadratic IPs with box constraints, presented in \cite{buchheim2013exact}, solves instances with up to 50 variables.
An SDP-based B\&B framework for nonconvex quadratic mixed-integer programming, introduced in~\cite{Buchheim2012SemidefiniteRF} and improved in~\cite{Buchheim2018SDPbasedBF}, solves
 instances of the nonconvex quadratic unconstrained ternary optimization  problem (QUTO) 
 up to 80 variables and TQP instances with one linear inequality up to 50 variables.
For the closely related max-cut problem, several SDP-based solvers exist, see e.g.,~\cite{HrgaPovh2021,Krislock2017BiqCrunch,gusmeroli2022biqbin}.   
The parallel version of the solver {MADAM} \cite{HrgaPovh2021} solves instances on graphs with  up to 180 vertices.
We note that general nonconvex mixed-integer programming solvers can also be used to solve nonconvex quadratic IPs, such as
{COUENNE}~\cite{Belotti2009Branching}, {BARON}~\cite{Sahinidis2016BARON}, and Gurobi~\cite{gurobi}.
However, these solvers are often outperformed on nonconvex TQP instances by the specialized solvers mentioned earlier, while they perform well on TQP instances with a convex objective function.

An SDP-based approximation algorithm for quadratic programs with ratio objective and ternary variables is derived in~\cite{bhaskara2012},
and an SDP-based randomized algorithm for the sparse integer least squares problem (SILS) was studied in~\cite{PiaZhou}. 
A heuristic  approach for the TQP with a circulant objective matrix is presented in~\cite{CallegariEtAll}.

We highlight several applications of optimization problems with ternary variables. 
Problems arising in engineering applications may be formulated as the TQP.
For instance, a variant of TQP with a single linear constraint appears in the design of turbomachines~\cite{Phuong}.
In electronics, the filtered approximation problem can be modeled as a convex QUTO~\cite{CallegariEtAll}.
Quadratic programs with ratio objectives and ternary variables arise when modeling density clustering problems with two voting options, including the possibility of abstention~\cite{Amaral2015CopositivitybasedAF}, 
in PageRank for folksonomies in social media \cite{hotho2006information},  in discovering conflicting groups in signed networks~\cite{NEURIPS2020_7cc538b1},  and when computing max-cut gain using eigenvalue techniques \cite{Trevisan2009MaxCut}.
The SILS problem finds applications in many fields, including cybersecurity, sensor networks,  digital fingerprints, array signal processing, and multiuser detection algorithms, see e.g.,~\cite{PiaZhou} and references therein.

{\color{black}We note here that one can also use the cube roots of unity to model ternary choices.
Since there is no gain in considering complex SDP programs when the objective matrix is real \cite{GOEMANS2004442,Sinjorgo2026}, we do not pursue this idea further.}

\subsection*{Main results and outline}

This paper studies theoretical and algorithmic approaches for the TQP, resulting in an SDP-based B\&B algorithm for general ternary quadratic programs. Our algorithm is tested on various special cases of the TQP inspired by practical applications, showing its effectiveness compared to general nonconvex mixed-integer programming solvers. 

Our approach starts with a theoretical analysis of the set of ternary PSD matrices. We recall some known results on integer PSD matrices and extend these with several new characterizations of the set of ternary PSD matrices. 
One of these characterization induces a ternary SDP formulation for the TQP. Relaxation of the integrality constraints yields an SDP relaxation, which serves as the foundation of our approach. We show that the relaxation can be strengthened by a large number of linear inequalities, including the triangle, $k$-gonal, RLT (i.e., reformulation linearization technique), standard split and non-standard split inequalities.

The basic SDP relaxation and the derived inequalities are combined in a specialized B\&B algorithm for solving nonconvex  TQP instances. 
The strengthened SDP bounds are incorporated in a ternary branching framework, 
where a Variable Neighborhood Search (VNS) heuristic is embedded to obtain high-quality feasible solutions throughout the algorithm. 
Our B\&B algorithm is tested on three special cases of the TQP. The first two cases correspond to the unconstrained version of the problem (i.e., QUTO) and the TQP with one linear constraint (i.e., TQP-Linear)~\cite{Phuong}. The third type of TQP involves the minimization of the ratio of two quadratic functions over the set $\{0,\pm 1\}^n$, which we denote by TQP-Ratio~\cite{Amaral2015CopositivitybasedAF, hotho2006information,Trevisan2009MaxCut}. 
These problems can be solved using a parametrized approach inspired by~\cite{Dinkelbach1967NonlinearFractional}, where each subproblem boils down to solving a QUTO. Alternatively, we also introduce an SDP relaxation for TQP-Ratio problems based on a reformulation of the objective function inspired by the convexification approach from~\cite{he2025convexification}.

In our computational analysis, we compare our B\&B algorithm with state-of-the-art commercial solvers, such as \texttt{GUROBI}~\cite{gurobi}. \textcolor{black}{We test our method on different types of randomly generated instances and on structured instances derived by graph applications.} Our results show that we significantly outperform~\texttt{GUROBI} on all three TQP variations. For the \textcolor{black}{randomly generated} instances with up to $90$ variables, our approach is able to solve all instances within a time limit of 3600 seconds, whereas~\texttt{GUROBI} fails on a substantial number of them. Also, for larger instances, our B\&B algorithm either solves instances to optimality or is able to produce high-quality solutions, with remaining gaps typically below about 2\%. \textcolor{black}{For the graph instances, the behaviour is similar, since we outperform \texttt{GUROBI} on both sparse and dense instances, and we scale better when the size of the graph increases.}
This study shows that our algorithm is currently the best known algorithm for solving nonconvex TQP instances.

This paper is organized as follows. Theoretical results on ternary PSD matrices and their characterizations are discussed in Section~\ref{sect: matrix properties}. In Section~\ref{Sec:ISDP}, we introduce the TQP and reformulate it as an SDP with integrality constraints. Its SDP relaxation and various families of inequalities are derived in Section~\ref{Sect:SDP and cuts}. In Section~\ref{Sec:SpecialCases}, we present three special cases of the TQP. Our B\&B algorithm is introduced in Section~\ref{Sec:BandB}. Finally, in Section~\ref{sect:numResults} numerical results are presented.

\subsection*{Notation}
For $n \in \mathbb{Z}_+$, we define the set $[n] := \{1, \ldots, n\}$.
We denote by $\mathbf{0}_n$ and $\mathbf{1}_n$ the vector of zeros and ones in $\mathbb{R}^n$, respectively.
The $n\times n$ matrix of ones is denoted by $\mathbf{J}_n$. The $n\times n$ identity matrix is denoted by $\mathbf{I}_n$.
We omit the subscripts of these matrices when their sizes are clear from the context. For fixed~$n$, matrices of size $(n+1)\times(n+1)$ 
and vectors of size $(n+1)\times 1$ are indexed  from zero.  
We denote by   $\mathbf{e}_i$  the $i^{\text{th}}$ unit vector. For any $S \subseteq [n]$, $\mathbbm{1}_S$ denotes the indicator vector induced by $S$.
We denote the set of all $n \times n$ real symmetric matrices by  $\mathcal{S}^n$, and  
  by $X_J$ the $|J| \times  |J|$ principal submatrix of $X\in \mathcal{S}^n$ with rows and columns indexed by $J\subseteq [n]$.
The cone of symmetric PSD matrices of order $n$ is defined as $\mathcal{S}^n_+ := \{ {X} \in \mathcal{S}^n \, : \, \, {X} \succeq \mathbf{0} \}$. 
The trace of a square matrix ${X}=(X_{ij})$ is given by $\tr({X})=\sum_{i}X_{ii}$.
 For any ${X},{Y} \in \mathcal{S}^n$ the trace inner product is defined as $\langle {X}, {Y} \rangle := \tr({X} {Y}) = \sum_{i,j = 1}^n X_{ij} Y_{ij}$. 
 For $x\in \mathbb{R}^n$ we denote by $|x|$ the vector whose elements are the absolute values of the corresponding elements of $x$, and by $\mbox{supp}(x) \subseteq [n]$ its support. 
 The operator $\diag : \mathbb{R}^{n \times n} \rightarrow \mathbb{R}^n$ maps a square matrix to a vector consisting of its diagonal elements. We denote by $\Diag : \mathbb{R}^n \rightarrow \mathbb{R}^{n \times n}$ its adjoint operator. 
 We use $\color{black}\mathrm{conv}(S)$ to denote the convex hull of the set $S$.

\section{Characterizations of PSD $\mathbf{\{0,\pm 1\}}-$matrices and related polyhedra}
\label{sect: matrix properties}

In this section, we present three characterizations of ternary PSD matrices, which are then
used to derive novel reformulations of polytopes from the literature \cite{10.1007/978-3-642-13036-6_20,Burer2014UnboundedCS}.
We also relate these polytopes with the well-known Boolean quadric polytope and the  set-completely positive matrix cone.

The following characterization of PSD $\{0,\pm 1\}-$matrices is derived in~\cite{DeMeijerSotirov23}. 
\begin{proposition}[\cite{DeMeijerSotirov23}]
\label{prop:Characterization}
    Let $X \in \{0, \pm 1\}^{n \times n}$ be symmetric. Then, $X \succeq \mathbf{0}$ if and only if $X = \sum_{i = 1}^r x_ix_i^\top$ for some $x_i \in \{0,\pm 1\}^n$, $i \in [r]$. 
\end{proposition}
Observe  that the vectors $x_i$ in the above proposition have disjoint supports. There exists a similar decomposition  of PSD  $\{0, 1\}-$matrices~\cite{Letchford2008BinaryPS}. 
The following  characterization  of ternary PSD matrices follows directly from the previous proposition. 
\begin{corollary}
A symmetric ternary matrix is positive semidefinite if and only if it equals $TT^\top$    for some ternary matrix $T$.
\end{corollary}

The following result from~\cite{DeMeijerSotirov23} relates a PSD $\{0,\pm 1\}-$matrix of rank one with an extended linear matrix inequality.
\begin{proposition}[\cite{DeMeijerSotirov23}]
\label{prop:Rank1}
Let $Y = \begin{pmatrix} 1 &  x^\top  \\ {x} & {X}\end{pmatrix} \in {\mathcal S}^{n+1}$  with $\diag(X)=|x|$.
Then,
 $Y \in \{0,\pm 1\}^{(n+1)\times (n+1)}$, $Y \succeq {\mathbf 0}$ if and only if $X=xx^\top.$
\end{proposition}
One can also replace    $\diag(X)=|x|$ by $\mbox{supp}(\diag(X))= \mbox{supp}(x)$ in \Cref{prop:Rank1}.

In the sequel, we first present known results that characterize PSD matrices with entries in $\{\pm1\}$ and $\{0,1\}$ by using inequalities, and then we derive a similar result for PSD $\{0,\pm 1\}-$matrices.
The following well-known result characterizes PSD $\{\pm 1 \}-$matrices using the triangle inequalities and a diagonal constraint.
\begin{proposition}[\cite{DezaLaurent1997}] \label{prop:characterMC}
    Let $X \in \{\pm 1\}^{n \times n}$ be symmetric. Then, $X = xx^\top$ for some $x \in \{\pm 1\}^n$ if and only if $X$ satisfies 
    \begin{align*} 
    X_{ij} +  X_{ik} + X_{jk} & \geq -1, \quad 
     -X_{ij} + X_{ik} - X_{jk}   \geq -1  & &  \forall 1\leq i < j < k \leq n  \\ 
   X_{ij} - X_{ik} - X_{jk}    & \geq -1, \quad 
     -X_{ij} - X_{ik} + X_{jk}   \geq -1 &&   \forall  1\leq i < j < k \leq n \\ 
        X_{ii} & = 1 && \forall i \in [n].
        \end{align*}         
\end{proposition}
One can replace the rank-one condition from  \Cref{prop:characterMC} by the positive semidefiniteness constraint. Namely, for a symmetric matrix $X$, we have that $X\succeq \mathbf{0}$, $X \in \{\pm 1\}^{n \times n}$ if and only if $X = xx^\top$ for some $x \in \{\pm 1\}^n$ \cite{GoemansWilliamson}.

A characterization  of PSD  $\{0, 1\}-$matrices  in terms of inequalities is given below.
\begin{proposition}[\cite{Letchford2008BinaryPS}]
Let $X \in \{0,1\}^{n\times n}$ be symmetric with $n\geq 3$.
Then, $X \succeq \mathbf{0}$ if and only if $X$ satisfies
\begin{align}
X_{ik} + X_{jk} &\leq X_{kk}+X_{ij} && \forall  1\leq i < j \leq n, \,\, k\neq i,j  \nonumber \\
X_{ij}&\leq X_{ii}  &&  \forall  1\leq i < j \leq n.  \nonumber 
\end{align}
\end{proposition}
For the case of ternary PSD matrices, we derive the following result.
\begin{theorem}\label{prop:triangle_ineq}
    Let $X\in\{0, \pm 1\}^{n\times n}$ be symmetric with $n\geq 3$. Then, $X \succeq \mathbf{0}$ if and only if $X$ satisfies 
    \begin{align}
     X_{ij} +  X_{ik} + X_{jk} & \geq -1, \quad 
     -X_{ij} + X_{ik} - X_{jk}   \geq -1  & &  \forall 1\leq i < j < k \leq n  \label{C:eq1} \\
   X_{ij} - X_{ik} - X_{jk}    & \geq -1, \quad 
     -X_{ij} - X_{ik} + X_{jk}   \geq -1 &&   \forall  1\leq i < j < k \leq n \label{C:eq2.3}\\
        X_{ij} &\leq X_{ii}, \quad   X_{ij} \geq - X_{ii} && \forall i, j \in [n], ~i \neq j. \label{C:eq4}
    \end{align}
\end{theorem}
\begin{proof}
$(\Rightarrow)$: Let $X \succeq \mathbf{0}$. Then, by Proposition~\ref{prop:Characterization} it follows that there exist $x_i \in \{0,\pm 1\}^n$ for $i \in [r]$ such that $X = \sum_{i =1}^r x_ix_i^\top$. 
Let $K_i = \mbox{supp}(x_i)$ for all $i \in [r]$ and define $K_0 = [n] \setminus (\bigcup_{i \in [r]}K_i)$. Since the vectors $x_i$ have disjoint support, $\{K_0, K_1, \ldots, K_r\}$ forms a partition of $[n]$. 

It follows from Proposition~\ref{prop:characterMC} that for each $i, j, k$ within the same set $K_\ell$ of the partition, the inequalities~\eqref{C:eq1}--\eqref{C:eq2.3} are satisfied. 
Now, let us assume that $i$ and $j$ are in the same subset of the partition, while $k$ is in a different subset. Then, $X_{ik} = X_{jk} = 0$, while $X_{ij}$ can be either $1$ or $-1$. One can easily check that~\eqref{C:eq1}--\eqref{C:eq2.3} are satisfied. Finally, if $i, j$ and $k$ are all in separate subsets of the partition, then $X_{ij} = X_{ik} = X_{jk} = 0$, which also satisfy~\eqref{C:eq1}--\eqref{C:eq2.3}. 

Let us now consider the inequalities~\eqref{C:eq4}. Since $X \succeq \mathbf{0}$, each diagonal element $X_{ii}$ is either $0$ or $1$. If $X_{ii} = 1$, then~\eqref{C:eq4} is clearly satisfied for all $j \in [n]$. If $X_{ii} = 0$, then $X_{ij} = 0$ for all $j \in [n]$ due to the fact that $X$ is PSD and thus,~\eqref{C:eq4} are also satisfied. We conclude that $X$ satisfies the inequalities~\eqref{C:eq1}--\eqref{C:eq4}.

$(\Leftarrow):$ Now, assume that $X$ satisfies~\eqref{C:eq1}--\eqref{C:eq4}. From~\eqref{C:eq4} it follows that $X_{ii} \in \{0,1\}$ for all $i \in [n]$. Let $K_0 := \{i \in [n] \, : \, \, X_{ii} = 0\}$. It follows from~\eqref{C:eq4} that $X_{ij} = X_{ji} = 0$ for  $i \in K_0$ and $j \in [n]$. Now, take some $i^* \in [n] \setminus K_0$ and define $K_1 := \{j \in [n] \, : \, \, X_{i^*j} \neq 0\}$. We show that the submatrix of $X$ induced by $K_1$ has entries in $\{\pm 1\}$. Clearly, this submatrix cannot have zeros on the diagonal. Now, suppose $X_{jk} = 0$ for some $j, k \in K_1$. Since $X_{i^* j}, X_{i^* k} \in \{\pm 1\}$, one can verify that at least one of the triangle inequalities is violated in case $X_{j k} = 0$. Thus, the submatrix of $X$ induced by $K_1$ contains only entries in $\{\pm 1\}$. 

Next, we show that $X_{jk} = 0$ for all $j \in K_1$, $k \notin K_1$. It follows from the definition of $K_1$ that $X_{i^* k} = 0$. Moreover, as argued above, we must have $X_{i^* j} \in \{\pm 1\}$. Suppose $X_{i^* j} = 1$. Then, $X_{jk}$ cannot be $1$, since then
 $   -X_{i^* j}  + X_{i^*k} - X_{jk} = -1 + 0 - 1 < -1,$
which violates one of the triangle inequalities~\eqref{C:eq1}--\eqref{C:eq2.3}.
Similarly, $X_{jk}$ cannot be $-1$, since then 
 $ -X_{i^* j} - X_{i^* k} + X_{jk} = -1 - 0 + (-1) < -1,$
also violating the triangle inequalities.
Thus, $X_{jk} = 0$. The proof for the case $X_{i^* j} = -1$ is similar. We conclude that $X_{jk} = 0$ for all $j \in K_1$, $k \notin K_1$. 

Now, we take $j^* \in [n] \setminus (K_0 \cup K_1)$ and define $K_2 := \{ j \in [n] \, : \, \, X_{j^* j} \neq 0\}$. By the same arguments as above, it follows that the submatrix of $X$ induced by $K_2$ has entries in $\{\pm 1\}$ and that $X_{ij} = 0$ for all $i \in K_2$, $j \notin K_2$. By iteratively repeating this argument, we obtain a partition $\{K_0, K_1, \ldots, K_r\}$ of $[n]$ that defines a non-overlapping block structure in $X$, where $K_0$ corresponds to a block of zeros and each $K_\ell$, $\ell \neq 0$, defines a block with entries in $\{\pm 1\}$. 
It follows from Proposition~\ref{prop:characterMC} that there exist $x_i \in \{0, \pm 1\}^n$ for $i \in [r]$ such that $X = \sum_{i=1}^r x_ix_i^\top$, hence $X \succeq \mathbf{0}$.
\end{proof}
The inequalities from \Cref{prop:triangle_ineq} can be used to strengthen SDP relaxations of TQPs.


\begin{remark}
Similar to the binary PSD matrices~\cite{Letchford2008BinaryPS}, the ternary PSD matrices have a graphical representation.
Let $X$ be a $n\times n$ rank-one ternary PSD and let $S = \{ i\in [n] \,:\, X_{ii} = 1 \}$.
Then, the submatrix of $X$ induced by $S$
corresponds to the edge cut-set of the subgraph of the complete graph of order $n$ induced by $S$. If $\rank(X) > 1$, then $X$ is the direct sum of $r$ non-overlapping edge cut-set matrices.
\end{remark}
In what follows we  exploit characterizations of PSD  $\{0,\pm 1\}$-matrices to derive equivalent formulations of polytopes from the literature.
Letchford~\cite{10.1007/978-3-642-13036-6_20} introduced and studied the sets 
$$
F_n := \left \{  (x,y) \in  \Z^{n + \binom{n+1}{2} } \, : \, y_{ij} = x_i x_j  \mbox{  for  }  1\leq i \leq j \leq n  \right \},
$$
and  $\mbox{IQ}_n := {\color{black}\mathrm{clconv} (  F_n )}$ for $n\in \mathbb{N}$. 
Here ${\color{black}\mathrm{clconv}(\cdot)}$ denotes the closure of the  convex hull of $F_n$. The set  $\mbox{IQ}_n$ is  not a polyhedron since it has a countably infinite number of faces. Such sets are known as  quasi-polyhedra. It was proven in~\cite{10.1007/978-3-642-13036-6_20} that every point in  $F_n$ is an extreme point of $\mbox{IQ}_n$. 
Moreover, it is not difficult to verify that  
{\color{black}$F_n$ coincides with the following set, apart from the embedding,}
$$
\left \{
X^+ = \begin{pmatrix}
    1 & x^\top \\
    x & X
\end{pmatrix} \,:\, X^+ \succeq \mathbf{0},  \,\,
\rank(X^+) = 1, \,\, x \in \Z^n
\right \}.
$$ 
Burer and Letchford~\cite{Burer2014UnboundedCS} proved that minimizing a linear function over $\mbox{IQ}_n$ is $\mathcal{NP}$-hard in the strong sense. 
If there exists  $u\in \Z_+$ such that $-u \mathbf{1}_n \leq x \leq u\mathbf{1}_n$ for all $(x,y) \in F_n$, then one can  define  the following bounded sets:
\begin{align*}
F^u_n := \left \{  (x,y) \in  \Z^{n + \binom{n+1}{2} }  \,:\,  y_{ij} = x_i x_j \,\,  
\mbox{  for  }
1\leq i \leq j \leq n, 
\,\, x_i \in \{-u,\ldots, u\} \,\, i \in [n] \right \}
\end{align*}
and  $\mbox{IQ}^u_n := {\color{black}\mathrm{conv}\left( F^u_n \right)}$. The set $\mbox{IQ}^u_n$ is a polytope.
The sets $\mbox{IQ}_n$ and  $\mbox{IQ}^u_n$ are studied in   \cite{Buchheim2015OnTS,Burer2014UnboundedCS,GALLI2021100661,10.1007/978-3-642-13036-6_20}.
For the case $u=1$, we define
\begin{align} \label{F1n}
\mathcal{F}_n^1 := \left \{ 
X \in \{0, \pm 1\}^{n\times n}
\,:\,
\begin{pmatrix} 1 &  x^\top  \\ {x} & {X}\end{pmatrix}  \succeq \mathbf{0}, \,\   \diag(X)=|x|
\right \}.
\end{align}
Note that, apart from the embedding, $\mathcal{F}_n^1$ coincides with $F_n^1$.
Elements of the set $\mathcal{F}_n^1$ are neither  described by a nonconvex rank constraint 
nor by a product of variables.  
The cardinality of $\mathcal{F}_n^1$ is $1+\sum_{k=1}^n\binom{n}{k}2^{k-1}$, see~\cite{DeMeijerSotirov25} for the enumeration of its vertices.
Let us consider the following two polytopes:
\begin{align} \label{IQ1n}
\mathcal{IQ}^1_n & := {\color{black}\mathrm{conv} \left( \mathcal{F}_n^1 \right)}, \\
\mathcal{M}_n & :=  \left \{ X\in \sn \, : \, \eqref{C:eq1}-\eqref{C:eq4}, \,\, X_{ii}\leq 1 \mbox{ for }  i\in [n]  \right \}. \nonumber
\end{align}
It follows from Proposition~\ref{prop:Rank1} that $\mathcal{IQ}^1_n$ is the convex hull of all rank-one ternary PSD matrices. 
{\color{black}One can easily show  that all ternary PSD matrices are also contained in $\mathcal{IQ}^1_n$.}
For example, any ternary PSD matrix $X$ of rank $2$ can be decomposed as follows: Since $X = x_1 x_1^\top + x_2 x_2^\top$ for some $x_1, x_2 \in \{0,\pm 1 \}^n$, we define $S_i := \{j : (x_i)_j = 1\}$ and $T_i := \{ {\color{black}j} : (x_i)_j = -1\}$ for $i \in \{1,2\}$. Then, we have
\begin{align*}
    X  =  \frac{1}{2}(\mathbbm{1}_{S_1 \cup S_2} - \mathbbm{1}_{T_1 \cup T_2}) (\mathbbm{1}_{S_1 \cup S_2} - \mathbbm{1}_{T_1 \cup T_2})^\top \hspace{-0.15cm}
    + \frac{1}{2}(\mathbbm{1}_{S_1 \cup T_2} - \mathbbm{1}_{S_2 \cup T_1}) (\mathbbm{1}_{S_1 \cup T_2} - \mathbbm{1}_{S_2 \cup T_1})^\top 
\end{align*} 
showing that $X \in \mathcal{IQ}^1_n$. A similar construction holds for higher rank matrices.

The set $\mathcal{M}_n$ includes constraints   \eqref{C:eq1}--\eqref{C:eq4}, and additionally upper bounds on
the diagonal elements. Note that without those upper bounds, the corresponding set would be unbounded.
Intrigued by the well-known result that relates the metric polytope and the cut polytope, 
we study the relation between $\mathcal{IQ}^1_n $ and $\mathcal{M}_n$.
Recall that the metric polytope is 
the set of all symmetric matrices with ones on the diagonal that satisfy the triangle inequalities
\eqref{C:eq1}--\eqref{C:eq2.3}, 
and the cut polytope  is ${\color{black}\mathrm{conv}\left(\{ X\in \sn : X=yy^\top, \,y\in \{\pm 1\}^n \}\right)}$.
The well-known result that relates these two polytopes is that the metric polytope coincides with the cut polytope for $n\in \{3, 4\}$, while the cut polytope is strictly contained in the metric polytope 
for $n\geq 5$, see e.g.,~\cite{DezaLaurent1997}.

Next, we  show that $\mathcal{IQ}^1_n \subset\mathcal{M}_n$ for $n=3$.
{\color{black}By exploiting convexity of $\mathcal{M}_3$ and $\mathcal{IQ}^1_n$, it follows that 
$\mathcal{IQ}^1_3 \subseteq \mathcal{M}_3$.}
To show that  the inclusion is strict, we consider the matrix
$ \tfrac{1}{3}
\begin{psmallmatrix}
      1 & 1 & -1 \\[0.7ex]
  1 & 1 & 1 \\[0.7ex]
 -1 & 1 & 3
\end{psmallmatrix} \in \mathcal{M}_3,
$
that is  one of the vertices of~$\mathcal{M}_3$.
This matrix is not positive semidefinite and it is therefore not in  $\mathcal{IQ}^1_3$.

\medskip
Next, we consider an extended formulation of the polytope $\mathcal{IQ}^1_n$, i.e., we write $\mathcal{IQ}^1_n$ as a linear projection of a lifted polytope in a higher-dimensional space. To do so, we start from the observation that any ternary vector $x \in \{0, \pm 1\}^n$ can be decomposed as $x = x' - x''$, where $x', x'' \in \{0,1\}^n$. This decomposition is unique if we additionally require that $x' + x'' \leq \mathbf{1}_n$.
Exploiting the decomposition, the ternary rank-one PSD matrix $xx^\top$ can be written as
\begin{align} \label{eq:extended_eq1}
    xx^\top = (x' - x'')(x' - x'')^\top = x'(x')^\top + x''(x'')^\top - x'(x'')^\top - x''(x')^\top.
\end{align} 
Now, consider the extended rank-one matrix $Z$ that is defined as
\begin{align} \label{eq:extended_eq2}
    Z = \begin{pmatrix}
        x' \\ x'' 
    \end{pmatrix}\begin{pmatrix}
        x' \\ x'' 
    \end{pmatrix}^\top = \begin{pmatrix}
        x'x'^\top & x'(x'')^\top \\
        x''(x')^\top & x''(x'')^\top 
    \end{pmatrix} =: \begin{pmatrix}
        Z^{11} & Z^{12} \\ Z^{21} & Z^{22}
    \end{pmatrix}. 
\end{align} 
Combining~\eqref{eq:extended_eq1} and~\eqref{eq:extended_eq2}, it follows that $xx^\top = Z^{11} + Z^{22} - Z^{12} - Z^{21}$. The discrete set $\mathcal{F}^1_n$ can therefore be written as the linear projection of a discrete set in $(2n+1)$-dimensional space. {Let the lifting $\mathcal{L}$ be defined as
\begin{align*}
   \mathcal{L}(z_1, z_2, Z^{11}, Z^{12}, Z^{21}, Z^{22}) := \begin{pmatrix}
            1 & z_1^\top & z_2^\top\\
            z_1 & Z^{11} & Z^{12} \\
            z_2 & Z^{21} & Z^{22}
        \end{pmatrix}.
\end{align*}
Then,}
        \begin{align} \label{eq:F1n-binary}
        \mathcal{F}^1_n = \left\{Z^{11} + Z^{22} - Z^{12} - Z^{21}~:~~ \begin{aligned} {\mathcal{L}(z_1, z_2, Z^{11}, Z^{12}, Z^{21}, Z^{22}) \succeq \mathbf{0}}, \\
        \diag(Z^{11}) = z_1,~
        \diag(Z^{22}) = z_2, \\
        z_1 + z_2 \leq \mathbf{1}_n,~ 
        z_1, z_2 \in \{0,1\}^n \end{aligned}\right\}. 
    \end{align}
Observe that it is sufficient to enforce integrality on $z_1, z_2$, as $\diag(Z^{11}) = z_1$ and $\diag(Z^{22}) = z_2$ together with the linear matrix inequality enforce integrality of the entire $(2n+1)$-dimensional matrix~\cite{helmberg2000semidefinite}. Moreover, as shown in~\cite{DeMeijerSotirov23}, each extended matrix in the set on the right-hand side of~\eqref{eq:F1n-binary} has {\color{black} rank one}. 
The description of $\mathcal{F}_n^1$ in~\eqref{eq:F1n-binary} shows that any ternary rank-one PSD matrix is the projection of a binary rank-one PSD matrix in a higher-dimensional space. The following proposition connects their induced polytopes in a similar way.

\begin{proposition} \label{Prop:IQ_extendedform}
Let $\mathcal{P}^1_n := {\color{black} \mathrm{conv}  (\{ X \in \{0,1\}^{n \times n} \, : \, \, X \succeq \mathbf{0},~\rank(X) = 1\}) }$, i.e., the convex hull of all binary rank-one PSD matrices. Then,
    \begin{align}\label{eq:IQlifted}
        \mathcal{IQ}^1_n = \left\{ Z^{11} + Z^{22} - Z^{12} - Z^{21}~ :~~ \begin{aligned}
        {\mathcal{L}(z_1, z_2, Z^{11}, Z^{12}, Z^{21}, Z^{22}) \in \mathcal{P}_{2n+1}^1}\\ 
        \diag(Z^{11}) = z_1,~ \diag(Z^{22}) = z_2, \\
        z_1 + z_2 \leq \mathbf{1}_n \end{aligned}
        \right\}. 
    \end{align}
\end{proposition}

\begin{proof}
    {\color{black}Let $\mathcal{G}$ denote the set on the right-hand side of~\eqref{eq:IQlifted}} 
    and let $\mathcal{D}^1_n :=  \{ X \in \{0,1\}^{n \times n}\, : \, \, X \succeq \mathbf{0},~\rank(X) = 1\}$, i.e., the set of binary rank-one PSD matrices. Then, the set on the right-hand side of~\eqref{eq:F1n-binary} can be written as
    \begin{align} \label{eq:F1n-binary2}
        \mathcal{F}^1_n = \left\{Z^{11} + Z^{22} - Z^{12} - Z^{21} \, : \, \, \begin{aligned}{\mathcal{L}(z_1, z_2, Z^{11}, Z^{12}, Z^{21}, Z^{22}) \in \mathcal{D}_{2n+1}^1} \\
        \diag(Z^{11}) = z_1, ~
        \diag(Z^{22}) = z_2, \\
        z_1 + z_2 \leq \mathbf{1}_n \end{aligned} \right\}, 
    \end{align}
    {\color{black} From~\eqref{eq:F1n-binary2} it readily follows that each matrix in $\mathcal{F}_n^1$ is also contained in $\mathcal{G}$. Since $\mathcal{G}$ is convex, it follows that $\mathcal{IQ}^1_n = \conv (\mathcal{F}^1_n) \subseteq \mathcal{G}$.

    It remains to prove that $\mathcal{G} \subseteq \mathcal{IQ}_n^1$. Let $X \in \mathcal{G}$, then there exist $z_1, z_2, Z^{11}, Z^{12}, Z^{21}$ and $Z^{22}$ such that
    \begin{align*}
        X = Z^{11} + Z^{22} - Z^{12} - Z^{21}, \quad & \mathcal{L}(z_1, z_2, Z^{11}, Z^{12}, Z^{21}, Z^{22}) \in \mathcal{P}_{2n+1}^1 \\
        \diag(Z^{11}) = z_1,  \quad & \diag(Z^{22}) = z_2, \quad  z_1 + z_2 \leq \bold{1}_n.
    \end{align*}
    From the definition of $\mathcal{P}^1_{2n+1}$ it follows that there exist binary vectors $w_1^i, w_2^i \in \{0,1\}^n$ and scalars $w^i_0 \in \{0,1\}$ and $\beta_i \in (0,1]$ for $i \in [r]$ with $\sum_{i=1}^r \beta_i = 1$, such that 
    \begin{align*}
        \begin{pmatrix}
            1 & z_1^\top & z_2^\top\\
            z_1 & Z^{11} & Z^{12} \\
            z_2 & Z^{21} & Z^{22}
        \end{pmatrix} = \sum_{i=1}^r \beta_i\begin{pmatrix}
            w^i_0 \\ w^i_1 \\ w^i_2
        \end{pmatrix} \begin{pmatrix}
            w^i_0 \\ w^i_1 \\ w^i_2
        \end{pmatrix}^\top.
    \end{align*}
    Due to the binarity of $w_0^i$, we have $\sum_{i=1}^r (w_0^i)^2\beta_i = \sum_{i=1}^r w_0^i\beta_i  = 1$. Since also $\sum_{i=1}^r \beta_i = 1$ and $\beta_i > 0$, we have $w_0^i = 1$ for all $i \in [r]$.
    Using the equation $X = Z^{11} + Z^{22} - Z^{12} - Z^{21}$, it follows that
    \begin{align*}
        X & = \sum_{i=1}^r \beta_i \left( w_1^i (w_1^i)^\top + w_2^i (w_2^i)^\top - w_1^i (w_2^i)^\top - w_2^i (w_1^i)^\top \right) \\
        & = \sum_{i=1}^r \beta_i \left((w_1^i - w_2^i)(w_1^i - w_2^i)^\top \right)
        = \sum_{i=1}^r \beta_i X_i, 
    \end{align*}
    where we define $X_i := (w_1^i - w_2^i)(w_1^i - w_2^i)^\top$ for all $i \in [r]$. The conditions $w_1^i, w_2^i \in \{0,1\}^n$ imply that $w_1^i - w_2^i \in \{0, \pm 1\}^n$ for all $i \in [r]$. Hence, $X_i \in \mathcal{F}_n^1$ for all $i \in [r]$, from where we conclude that $X = \sum_{i=1}^r \beta_iX_i \in \conv(\mathcal{F}_n^1) = \mathcal{IQ}_n^1$.   
    }
\end{proof}
{\color{black}
\begin{remark}
    Observe that the constraints $\diag(Z^{11}) = z_1$ and $\diag(Z^{22}) = z_2$  are redundant in $\mathcal{G}$, since these are implied by the constraint $\mathcal{L}(z_1, z_2, Z^{11}, Z^{12}, Z^{21}, Z^{22}) \in \mathcal{P}_{2n+1}^1$. However, since we exploit them later in this section, we keep them in the formulation.
\end{remark}
}
Proposition~\ref{Prop:IQ_extendedform} extends any known representation of the convex hull of binary rank-one PSD matrices into a representation of $\mathcal{IQ}^1_n$. It is not difficult to verify that the convex hull of binary rank-one PSD matrices relates to the well-known Boolean quadric polytope~\cite{Padberg1989TheBQ}, which is defined as $\mathcal{BQP}_n = {\color{black}\mathrm{conv} \big( \{ (x,y) \in \{0,1\}^{n + \binom{n}{2} }\, : \, \, y_{ij} = x_ix_j}$ ${\color{black} \text{for } 1 \leq i < j \leq n\}\big)}$.
Apart from the embedding, the set $\mathcal{P}_n^1$ coincides with $\mathcal{BQP}_n$, see also~\cite{DeMeijerSotirov23}. More specifically, if we define the symmetric vectorization operator $\text{svec}^+: \mathcal{S}^{n+1} \rightarrow \mathbb{R}^{n + \binom{n}{2}}$, which maps a symmetric matrix to the vector containing its strictly upper triangular entries, i.e., 
\begin{align*}
    \text{svec}^+\left( \begin{pmatrix}
    1 & z^\top \\ z & Z
\end{pmatrix} \right) := (z, y) \text{ where } y_{ij} = Z_{ij} \text{ for } 1 \leq i < j \leq n,
\end{align*}
it follows that $\text{svec}^+(\mathcal{P}^1_n) = \mathcal{BQP}_n$. Now, it follows from Proposition~\ref{Prop:IQ_extendedform} that
\begin{align*}
    \mathcal{IQ}^1_n = \left\{ Z^{11} + Z^{22} - Z^{12} - Z^{21}  : \begin{aligned}
        {\text{svec}^+ \left( \mathcal{L}(z_1, z_2, Z^{11}, Z^{12}, Z^{21}, Z^{22})\right) \in {\color{black}\mathcal{BQP}_{2n+1} }}\\
        \diag(Z^{11}) = z_1,~ \diag(Z^{22}) = z_2,~
        z_1 + z_2 \leq \mathbf{1}_n \end{aligned}
        \right\}. 
\end{align*} 
Alternatively, it is also known that $\mathcal{P}_n^1$ as defined in Proposition~\ref{Prop:IQ_extendedform} possesses a set-completely positive reformulation~\cite{Lieder2015} in terms of the cone of set-completely positive matrices $\mathcal{SCP}_n = {\color{black}\mathrm{conv} \big( \{ xx^\top \, : \, \, x \in \mathbb{R}^n_+,~x_1 \geq x_i  \mbox{  for all } i = 2, \ldots, n\} \big)}$. 
Combining a result from~\cite{Lieder2015}, see also~\cite[Theorem 5]{DeMeijerSotirov23}, with Proposition~\ref{Prop:IQ_extendedform}, it follows that
\begin{align*}
    \mathcal{IQ}^1_n = \left\{  Z^{11} + Z^{22} - Z^{12} - Z^{21} ~: ~~ \begin{aligned}
        { \mathcal{L}(z_1, z_2, Z^{11}, Z^{12}, Z^{21}, Z^{22}) \in {\color{black}\mathcal{SCP}_{2n+1}}} \\
        \diag(Z^{11}) = z_1,~ \diag(Z^{22}) = z_2, ~ z_1 + z_2 \leq \mathbf{1}_n  \end{aligned}
        \right\}. 
\end{align*}

\section{ISDP formulation for  the ternary quadratic problem}
\label{Sec:ISDP}
We first introduce the   quadratic optimization problem of interest and then reformulate it as an integer semidefinite program (ISDP), using the results from 
Section~\ref{sect: matrix properties}.

Let $Q = (q_{ij}) \in \sn$,   $c\in \R^n$, and
 $a_i\in \R^n$, $b_i \in \R$ for all $i\in [p]$, where $p\in \N$.
We consider the ternary quadratic problem (TQP): 
\begin{align} \label{ternaryProblem}
\begin{aligned}
\min\limits_{ x\in \{0, \pm 1 \}^n} ~~ &  x^\top Q x + c^\top x & \\
\st ~~& a_i^\top x = b_i  \quad i\in[p]. 
\end{aligned}
\end{align}
Although the TQP \eqref{ternaryProblem} has no quadratic constraints, the approach described here  applies also if we add quadratic constraints.
The quadratic objective in \eqref{ternaryProblem} can be written as $\langle Q, X \rangle  + c^\top x$, where we substitute $X$ for $xx^\top $. 
This yields the following exact reformulation of  \eqref{ternaryProblem}:
\begin{align} \label{ternaryProblemRank1}
\begin{aligned}
\min\limits_{ x\in \{0, \pm 1 \}^n} ~~ &  \langle Q, X \rangle  + c^\top x  \\
\st ~~  & a_i^\top x = b_i   \quad  i\in[p] \\
& Y = \begin{pmatrix}
1 & {x}^\top \\ {x} & {X}
\end{pmatrix} \succeq \mathbf{0},  \quad  \rank(Y)=1.  
\end{aligned}
\end{align}
The above problem  contains a nonconvex rank-one constraint.  However, one can utilize an
alternative notion of exactness in terms of integrality by exploiting \Cref{prop:Rank1}, and obtain the following
 ternary semidefinite program:
\begin{align} \label{ternaryProblemISDP}
\begin{aligned}
\min ~~ &  \langle Q, X \rangle  + c^\top x  \\
\st ~~ & a_i^\top x = b_i   \quad  i\in[p] \\
& Y = \begin{pmatrix}
1 & {x}^\top \\ {x} & {X}
\end{pmatrix} \succeq \mathbf{0}, \quad \diag(X) = |x| \\
&  Y\in \{0, \pm 1 \}^{(n+1) \times (n+1)}.
\end{aligned}
\end{align}
{\color{black}Note that an integer semidefinite program in standard primal form has a linear objective function,  linear constraints, and  an integrality restriction on a positive semidefinite matrix variable, see e.g.,~\cite{de2022chv}. 
In the ternary SDP, we allow the additional nonlinear constraint $\operatorname{diag}(X)=|x|$.} One can also reformulate \eqref{ternaryProblemRank1} as a binary semidefinite program by exploiting the 
transformation of ternary matrices into binary matrices, see \eqref{eq:F1n-binary}. However, such decomposition increases the size of the problem and makes it more difficult to solve.
The next result follows  from the previous discussion. 
\begin{theorem}
Optimization problems \eqref{ternaryProblem} and \eqref{ternaryProblemISDP} are equivalent.    
\end{theorem}
One can relax some of the integrality conditions on $Y$ in \eqref{ternaryProblemISDP} and maintain exactness. Namely, we can request only $x\in \{0,\pm 1 \}$ instead of  $Y\in \{0, \pm 1 \}^{(n+1) \times (n+1)}$ in \eqref{ternaryProblemISDP}. To see this, we consider the three-by-three principal submatrices of $Y$ indexed by ${\mathcal I} \times {\mathcal I}$, where  ${\mathcal I}=[0,i,j]^\top$, for some $i,j\in [n]$, $i<j$.
Here, $0$  represents the first row/column of $Y$. Then,
${\rm det} (Y_{\mathcal{I}})= -(x_i x_j - X_{ij})^2 \geq 0,$
from where it follows $ X_{ij}=x_i x_j.$
Here, we exploit the fact that $|x_i|=x_i^2$ for $x_i\in \{0,\pm 1\}.$
A more compact formulation of \eqref{ternaryProblemISDP} can be obtained by using the following result.
\begin{lemma} \label{lemma: compact representation}
Let  
 $Y = \begin{pmatrix}
1 & {x}^\top \\ {x} & {X}
\end{pmatrix} \succeq \mathbf{0} $ such that $\diag(X) = |x|$ and  $x\in \{0,\pm 1 \}$, and let
 $S=\sum_{i=1}^p \begin{psmallmatrix}  -b_i \\ a_i \end{psmallmatrix}  \begin{psmallmatrix}  -b_i \\ a_i \end{psmallmatrix}^\top $.
 Then,  $ a_i^\top x = b_i$  $\forall i\in[p]$ if and only if $\langle S,Y \rangle =0$.
\end{lemma}
\begin{proof}
It follows from  \Cref{prop:Rank1}   and the previous discussion,  that 
$Y=\begin{psmallmatrix}  1 \\ x \end{psmallmatrix} \begin{psmallmatrix}  1 \\ x \end{psmallmatrix}^\top $.
Therefore, $\sum_{i=1}^p \left \langle  \begin{psmallmatrix}  -b_i \\ a_i \end{psmallmatrix}  \begin{psmallmatrix}  -b_i \\ a_i \end{psmallmatrix}^\top, \begin{psmallmatrix}  1 \\ x \end{psmallmatrix} \begin{psmallmatrix}  1 \\ x \end{psmallmatrix}^\top \right \rangle = \sum_{i=1}^p ( a_i^\top x - b_i)^2,$
from where the result follows.
\end{proof}

\section{SDP relaxation and valid inequalities}
\label{Sect:SDP and cuts}
In this section, we first present a basic SDP relaxation for the TQP \eqref{ternaryProblem},
and then introduce several classes of valid inequalities that can be imposed to strengthen it. 
We derive those inequalities by using various techniques, including  the RLT and Chv\'atal-Gomory procedure.

We derive a semidefinite relaxation for the TQP \eqref{ternaryProblem} 
from the ISDP  reformulation~\eqref{ternaryProblemISDP} by relaxing integrality constraints and 
the  constraint $\diag(X)=|x|$.
The resulting basic SDP relaxation is given below. 
\begin{align} \label{ternaryProblemSDP}
\begin{aligned}
\min ~~ &  \langle Q, X \rangle  + c^\top x  \\
\st ~~ & a_i^\top x = b_i   \quad  i\in[p] \\
& \diag(X) \geq x, \quad  \diag(X) \geq -x \\
& \diag(X) \leq {\mathbf 1}_n, \quad 
 \begin{pmatrix}
1 & {x}^\top \\ {x} & {X}
\end{pmatrix} \succeq \mathbf{0}.
\end{aligned}
\end{align}
We add constraints that impose upper bounds on the diagonal of the matrix variable to preserve the boundedness of the relaxation.
In particular, one can add  to \eqref{ternaryProblemSDP}  the following constraints
\begin{align}\label{constr:quadratic}
\langle a_ia_i^\top, X \rangle = b_i^2 \quad \mbox{ for all } i\in [p].
\end{align}
Note that by adding  those constraints the corresponding linear constraints are, in general,  not redundant. 
Further, one can easily verify  that for $(x,X)$ such that 
$ X - xx^\top  \succeq \mathbf{0}$,
 $\diag(X) \geq x$, $\diag(X) \geq -x$ and $\diag(X) \leq {\mathbf 1}$ it follows that $-1 \leq X_{ij} \leq 1$ and   $-1 \leq x_i \leq 1$ for all $i,j\in [n].$ 

\begin{remark}\label{remark2}
Buchheim and Wiegele~\cite{Buchheim2012SemidefiniteRF}  consider a nonconvex quadratic mixed-integer optimization problem and derive its exact reformulation in matrix form by using a rank-one constraint. 
Their exact matrix reformulation, when restricted to the  TQP~\eqref{ternaryProblem}, differs from \eqref{ternaryProblemRank1} in the constraint   
$x\in \{0, \pm 1 \}^n$, which is replaced by  $(x_i, X_{ii}) \in  P(\{0,\pm 1 \})= {\color{black}\mathrm{conv} (\{ (x,x^2) ~:~x\in \{0,\pm 1 \} \})}$  for all $i\in [n]$.
The polytope $P(\{0,\pm 1 \})$ is completely described by the following lower bounding facets
\begin{align} \label{lowerFacet}
X_{ii} \geq x_i \quad \mbox{and} \quad X_{ii} \geq -x_i \quad i\in [n],
\end{align}
and the upper bounding facet
$X_{ii} \leq 1$ for $i\in [n].$
Therefore, the SDP relaxation of the TQP~\eqref{ternaryProblem} from  \cite{Buchheim2012SemidefiniteRF},  obtained by omitting the rank-one constraint, corresponds to~\eqref{ternaryProblemSDP}.
\end{remark}

In what follows, we introduce valid inequalities, which together with \eqref{constr:quadratic}, can be to added  the SDP relaxation~\eqref{ternaryProblemSDP}. The most natural choice is to add to~\eqref{ternaryProblemSDP} the inequalities \eqref{C:eq1}--\eqref{C:eq4}  from \Cref{prop:triangle_ineq}. The inequalities~\eqref{C:eq4} are known in the literature as the \emph{pair inequalities}, see e.g.,~\cite{piccialli2022sos}.
These inequalities are non-standard split inequalities, and we show how to derive them later. \\

\noindent {\bf The (generalized) triangle inequalities.} 
 The inequalities~\eqref{C:eq1}--\eqref{C:eq2.3} are known in the literature as the  triangle inequalities. 
These inequalities are also Chv\'atal-Gomory cuts for the ISDP~\eqref{ternaryProblemISDP}, as we will show below.
The Chv\'atal-Gomory (CG) cuts form a well-known family of cutting planes for integer programs, initially designed for integer linear programs by Chv\'atal~\cite{Chvatal} and Gomory~\cite{Gomory}. Cuts for integer conic problems and ISDP have been investigated in \cite{CezikIyengar} and in \cite{de2022chv,DeMeijerSotirovWiegele24}, respectively.

One can show that the first set of inequalities in \eqref{C:eq1} are Chv\'atal-Gomory cuts by taking the vector $v \in \Z^{n+1}$ (indexed from zero) which has all zero entries except at positions $i,j,k$ where $v_i=v_j=v_k=1$. Then, for any $Y$  feasible for \eqref{ternaryProblemISDP} we have
$v^\top Y v \geq 0$ since $Y, vv^\top \succeq \mathbf{0}$. 
Thus, $ 2Y_{ij}+2Y_{ik}+2 Y_{jk} \geq -Y_{ii}-Y_{jj}-Y_{kk} \geq -3$, from where it follows that 
$Y_{ij}+Y_{ik}+ Y_{jk} \geq \lceil -1.5 \rceil = -1$ where we exploited the fact that the left-hand side of the inequality is an integer number. 
Similarly, one can derive other triangle inequalities from \eqref{C:eq1}--\eqref{C:eq2.3} as  Chv\'atal-Gomory cuts.

Interestingly, triangle inequalities  \eqref{C:eq1}--\eqref{C:eq2.3} can be generalized for any $S\subseteq [n]$ such that $|S|$ is an {\em odd} number.
In particular, let $S\subseteq [n]$ and $|S|$ odd, then
we introduce the following valid inequalities
\begin{align} \label{oddS}
\sum_{i,j\in S, i<j} v_iv_j X_{ij} \geq \left \lceil -\frac{1}{2}|S| \right \rceil,
\end{align}
where $v_i,v_j \in \{\pm 1\}$. There are $2^{|S|-1}$ constraints of this type  for a given $S$.
To prove that  \eqref{oddS} are Chv\'atal-Gomory cuts for the ISDP~\eqref{ternaryProblemISDP}, one can take 
 the vector $v \in \Z^{n+1}$ (indexed from zero) 
with all entries equal to zero except for $1$ or $-1$ at positions corresponding to the elements of $S$. Then, one can proceed similarly as for deriving the triangle inequalities. Thus, for $|S|=5$ (resp.~$|S|=7$) we obtain pentagonal (resp.~heptagonal)   inequalities for the TQP, and so on.
 Note that for subsets with an even number of elements, the corresponding inequalities are not binding.

{\color{black} We remark that the triangle inequalities for the max-cut problem are a special case of hypermetric  inequalities, see e.g.,~\cite{DezaLaurent1997}.
 One can derive a hypermetric inequality by considering an integer vector $b$ such that $\mathbf{1}_n^\top b$ is odd. This condition ensures that $| x^\top b | \geq 1$ for all $x\in\{\pm1\}^n$, which implies that a valid hypermetric inequality  $\langle bb^\top, xx^\top \rangle \geq 1$ holds.  The triangle inequalities for the max-cut problem are obtained by choosing  $b\in \{ 0,\pm 1 \}^n$ with exactly three nonzero entries.   The pentagonal and heptagonal inequalities are obtained similarly.
Note that the valid inequalities \eqref{oddS}  for the TQP  cannot be derived as  hypermetric inequalities. }\\

\noindent {\bf The RLT inequalities.}
The SDP relaxation  \eqref{ternaryProblemSDP} can be strengthened by adding 
valid inequalities  obtained by using the reformulation linearization technique~\cite{Sherali1998ART}. In particular, we derive the following RLT inequalities considering the bounds on the original variables: 
\begin{align}
    & \mbox{ From $x_i\geq -1$ and $x_j \geq -1$ for $i,j\in [n]$, $i<j$ : } & \!\!\!  X_{ij} + x_i+x_j \geq -1.  \label{rlt:1} \\[0.7ex]
    & \mbox{ From $x_i\leq 1$ and $x_j \leq 1$ for $i,j\in [n]$, $i<j$ : } &  X_{ij} - x_i- x_j \geq -1. \label{rlt:2}  \\[0.7ex]
    & \mbox{ From $x_i\geq -1$ and $x_j \leq 1$ for $i,j\in [n]$, $i<j$ : } &   -X_{ij} + x_i- x_j \geq -1. \label{rlt:3} \\[0.7ex]
     & \mbox{ From $x_i\leq 1$ and $x_j \geq -1$ for $i,j\in [n]$, $i<j$ :} &   -X_{ij} - x_i + x_j \geq -1.  \label{rlt:4}
\end{align}
Inequalities of this kind were originally introduced by McCormick~\cite{McCormick1976ComputabilityOG}. It is interesting to note that the RLT inequalities \eqref{rlt:1}--\eqref{rlt:4} may be derived by applying the  Chv\'atal-Gomory procedure with appropriate dual multipliers. 
For example, to derive \eqref{rlt:1} one can take
$v\in \Z^{n+1}$   (indexed from zero) 
with all entries equal to zero except at positions $0,i,j$ where $v_0=v_i=v_j=1$ for $i< j$. 
Then,  from $\langle vv^\top,Y\rangle \geq 0$, where $Y$ is feasible for \eqref{ternaryProblemISDP}, and with appropriate rounding, \eqref{rlt:1} follows.
Similarly, one can derive \eqref{rlt:2}--\eqref{rlt:4} as Chv\'atal-Gomory cuts.
Thus, the RLT-inequalities~\eqref{rlt:1}--\eqref{rlt:4} are also Chv\'atal-Gomory cuts.  
Note that inequalities  \eqref{rlt:1}--\eqref{rlt:4} can be generalized to any $S\subseteq [n]$ such that  $|S|$ is even. \\

\noindent {\bf Split inequalities.}
Split inequalities are introduced by Letchford \cite{10.1007/978-3-642-13036-6_20}, and further studied in 
\cite{Burer2014UnboundedCS,Buchheim2015OnTS}.
A split inequality corresponds  to split disjunction of the form 
\begin{align}\label{constr:split}
(w^\top x \leq s) \, \lor \, (w^\top x \geq s+1),    
\end{align}
where $w\in \Z^n$ and $s \in \Z.$  The above disjunction imposes that points in the feasible region lie either in the halfspace $w^\top x \leq s$ or in the halfspace $w^\top x \geq s+1$.
The  split disjunction \eqref{constr:split} can be reformulated as the  quadratic inequality $(w^\top x - s)(w^\top x - s-1 )\geq 0$.
In the lifted space, that quadratic inequality is linearized and takes the form:
\begin{align} \label{split inequality}
\left \langle v(v+e_0)^\top, \begin{pmatrix} 1 &  x^\top  \\ {x} & {X}\end{pmatrix} \right \rangle \geq 0 \quad
\forall v \in \Z^{n+1},
\end{align}
where $v = (-s-1, w^\top)^\top \in \Z^{n+1}$, $e_0$ is the unit vector indexed from zero.
 Buchheim and Traversi~\cite{Buchheim2015OnTS}  tested the quality of split inequalities on problems with up to 50 variables and concluded that these cuts are strong. The main obstacle in applying these cuts is identifying the violated ones.  
Moreover, it is unclear whether the separation of split inequalities {\color{black}is in ${\mathcal P}$.}

{In this work, we exploit  two types of split inequalities, which we list below.}  It is not difficult to verify that the inequalities~\eqref{lowerFacet} are  split inequalities. To verify this claim for $X_{ii}\geq x_i$ (resp.~$X_{ii}\geq -x_i)$, take $v \in \Z^{n+1}$ (indexed from zero) that has all zero entries except $v_0=-1$ and $v_i=1$ (resp.~$v_i=-1$) for some $i\in[n]$.  
Then, the inequalities~\eqref{lowerFacet} follow by using the vectors in~\eqref{split inequality}.

A generalization of inequalities~\eqref{lowerFacet} to split inequalities with two indices $i,j \in [n]$, $1\leq i< j \leq n$ is given below:
\begin{align}
X_{ii} + X_{jj} +2X_{ij} + x_i+ x_j &\geq 0, \quad   X_{ii} + X_{jj} +2X_{ij} - x_i- x_j \geq 0  \label{newsplit1}\\
X_{ii} + X_{jj} -2X_{ij} + x_i- x_j &\geq 0,  \quad  X_{ii} + X_{jj} -2X_{ij} - x_i+ x_j \geq 0.  \label{newsplit4}
\end{align}
We derive the first split inequality in \eqref{newsplit1}
by taking a binary vector $v \in \{0,1\}^{n+1}$ (indexed from zero) with all entries equal to zero except at positions $i$ and $j$ ($i<j$), 
and applying it in~\eqref{split inequality}. The remaining inequalities can be derived in a similar way.
Clearly, this approach can be extended to  split inequalities with three or more indices. However, the number of associated valid inequalities increases exponentially. That is, for a given $k\leq n$ indices, there are $2^k$ split inequalities of the above type. 
We note that we tested additional types of split inequalities, but the two mentioned above were the most effective in our numerical experiments.
\\

\noindent {\bf Non-standard split inequalities.}
{\color{black}
Inequalities of the form $\langle vw^\top , Y \rangle \geq 0$ that are not of the form~\eqref{split inequality} are studied in~\cite{10.1007/978-3-642-13036-6_20}. These non-standard split inequalities are also derived as the product of two valid inequalities, but are different from the standard split disjunction of the form~\eqref{constr:split}. That is, the non-standard split disjunction is induced by $a^\top x \geq b$ and $c^\top x \geq d$, where $a$ and $b$ are different, such that all non-negative integer vectors either satisfy both or neither inequalities. Non-standard split inequalities are further studied in \cite{Burer2014UnboundedCS,Buchheim2015OnTS}.
}

The following result, which holds for the polytope  $\mathcal{IQ}^1_n$, see~\eqref{IQ1n}, was established in \cite{Buchheim2015OnTS}.
\begin{lemma}[\cite{Buchheim2015OnTS}] \label{lemma:non-standard split}
For any vector $v\in \Z^{n+1}$ with $v_0=0$, the non-standard split inequalities
$\left \langle v(v+e_j)^\top, \begin{psmallmatrix} 1 &  x^\top  \\ {x} & {X}\end{psmallmatrix} \right \rangle \geq 0,$ for $j\in [n]$, are valid for $X\in \mathcal{IQ}^1_n$. 
\end{lemma}

One can prove that the inequalities in~\eqref{C:eq4} are non-standard split inequalities. Namely, for the first set of inequalities in \eqref{C:eq4} (resp.~the second set in \eqref{C:eq4}),  take $v \in \Z^{n+1}$ (indexed from zero) with all entries zero except at position $i\in [n]$, where   $v_i=-1$ (resp.~$v_i=1$),  
 and the $j^{\text{th}}$ unit vector $e_j \in \Z^{n+1}$  (indexed from zero, $i\neq j$), and apply \Cref{lemma:non-standard split}.

 We used the software SageMath \cite{SageMath} to find the complete description of $\mathcal{IQ}^1_3$. The polytope  $\mathcal{IQ}^1_3$ is described by the triangle inequalities \eqref{C:eq1}--\eqref{C:eq2.3}, 
 bound constraints  $X_{ii}\leq 1$ for $i\in [3]$ (see also \Cref{remark2}), non-standard split inequalities of the form~\eqref{C:eq4}, and 12 additional non-standard split inequalities. In total, there are 31 facet-defining inequalities  of $\mathcal{IQ}^1_3$.
All these inequalities can be imposed on any principal submatrix $X_{J}$ of $X$ from \eqref{ternaryProblemSDP}, where
$J\subseteq [n]$, $|J|=3$, to strengthen the SDP relaxation.
Preliminary numerical results show that it is most beneficial adding the inequalities~\eqref{C:eq4}.

\section{Special cases of the TQP} \label{Sec:SpecialCases}
\subsection{The quadratic unconstrained ternary optimization problem}
\label{sect:QUTO}
In this section, we study the QUTO and its relation to the max-cut problem.

The QUTO  is given as follows:
\begin{align} \label{QUTO}
\begin{aligned}
\min ~~ &  \langle Q, X \rangle  + c^\top x  \\
\st ~~ &  \begin{pmatrix}
1 & {x}^\top \\ {x} & {X}
\end{pmatrix} \succeq \mathbf{0}, \quad \diag(X) = |x| \\
&  X \in \{0, \pm 1 \}^{(n+1) \times (n+1)}.
\end{aligned}
\end{align}
The well-known max-cut problem is a special case of the  QUTO~\eqref{QUTO}, in which the variables are restricted to take values in $\{\pm 1 \}$.
In the sequel, we show that the QUTO  also reduces to the max-cut problem for a particular quadratic objective.
\begin{theorem} \label{Thm:fixing}
Consider the QUTO  \eqref{QUTO}.
If $q_{ii} \leq 0$ for some $i\in[n]$, then there exists an optimal solution to this problem with 
$x_i \in \{\pm 1\}$.
\end{theorem} 
\begin{proof}
    Let $i$ be such that $q_{ii} \leq 0$. Now, suppose $x^*$ is a feasible solution to the unconstrained problem~\eqref{ternaryProblem} with $x^*_i = 0$. We define the sets $S := \{j : x_j^* = 1\}$ and $T:= \{j : x_j^* = -1\}$ and let $f(x^*) = (x^*)^\top Q x^* + c^\top x^*$.  Moreover, we define $\beta = 2\sum_{j \in S}q_{ij} - 2\sum_{j \in T}q_{ij} + c_i$.

    Now, we construct another solution $x'$ from $x^*$ such that $x'_i = 1$ and $x'_j = x^*_j$ for all $j \neq i$.  The objective value of this solution equals $f(x') = f(x^*) + q_{ii} + \beta$. Alternatively, we construct $x''$ from $x^*$ such that $x''_i = -1$ and $x''_j = x^*_j$ for all $j \neq i$. Then, $f(x'') = f(x^*) + q_{ii} - \beta$.
    Since $q_{ii} \leq 0$, it follows that either $f(x') \leq f(x^*)$ or $f(x'') \leq f(x^*)$. Restricting the $i$th element of $x$ to an element in $\{\pm 1\}$ will therefore never increase the objective value. This holds in particular for {\color{black}any} optimal solution.
\end{proof}

The result of Theorem~\ref{Thm:fixing} is exploited in the following corollary, which states that the 
QUTO is equivalent to the max-cut problem if $\diag(Q) \leq \mathbf{0}_n$. 
\begin{corollary}\label{corr:maxcut}
    Suppose $q_{ii} \leq 0$ for all $i \in [n]$. Then, the QUTO~\eqref{QUTO} is equivalent to the max-cut problem. 
\end{corollary}
\begin{proof}
    It follows from Theorem~\ref{Thm:fixing} that under the given conditions we can always find an optimal solution $x$ in the set $\{\pm 1\}^n$.  Therefore, the problem~\eqref{QUTO} is equivalent to the minimization version of the quadratic unconstrained binary optimization problem (QUBO)  with variables in~$\{\pm 1\}$, also known as the Ising QUBO, see e.g.,~\cite{punnen2022qubo}. The Ising QUBO, both in its minimization and maximization version, is equivalent to the max-cut problem, see e.g.,~\cite[Section 1.4.2]{punnen2022qubo}.  
\end{proof}

Our basic SDP relaxation for the QUTO is as follows:
\begin{align} \label{QUTOSDPbasic}
\begin{aligned} 
\min ~~ &  \langle Q, X \rangle  + c^\top x  \\
\st ~~  & X_{ii} = 1, \quad i\in N \\
& \diag(X) \geq x, \quad 
 \diag(X) \geq -x \\
& \diag(X) \leq {\mathbf 1}_n, \quad
 \begin{pmatrix}
1 & {x}^\top \\ {x} & {X}
\end{pmatrix} \succeq \mathbf{0},
\end{aligned}
\end{align}
where  $N = \{i \, : \, \, q_{ii} \leq 0\}$. 
One can  strengthen the SDP relaxation  \eqref{QUTOSDPbasic} by adding valid inequalities introduced in Section~\ref{Sect:SDP and cuts}.

\subsection{Ternary quadratic problem with one  linear constraint}
We study a variant of the TQP \eqref {ternaryProblem} with a single linear constraint, which we refer to as TQP-Linear. 
In particular, we present our basic SDP relaxation for the TQP-Linear and derive its strictly feasible formulation. 


The TQP in which a linear constraint enforces that the sum of the variables equals zero, is used in the design of turbomachines, see e.g.,~\cite{Phuong}. 
Inspired by this application, we consider the TQP \eqref{ternaryProblem} in which $a_1={\mathbf 1}_n$, $b_1=0$ and  $p=1$, thereby resulting in
the following TQP-Linear:
\begin{align} \label{turbine Problem}
\begin{aligned}
\min ~~ &  \langle Q, X \rangle  + c^\top x  \\
\st ~~ & {\mathbf 1}_n^\top x = 0, \quad 
\begin{pmatrix}
1 & {x}^\top \\ {x} & {X}
\end{pmatrix} \succeq \mathbf{0}, \quad \diag(X) = |x| \\
&  X \in \{0, \pm 1 \}^{(n+1) \times (n+1)}.
\end{aligned}
\end{align}
Our basic SDP relaxation for the TQP-Linear is as follows:
\begin{align} \label{ternaryProblemSDPbasic}
\begin{aligned} 
\min ~~ &  \langle Q, X \rangle  + c^\top x  \\
\st ~~ 
&\langle {\mathbf J},X\rangle = 0, \quad
 \diag(X) \geq x, \quad 
 \diag(X) \geq -x \\
&\diag(X) \leq {\mathbf 1}_n, \quad  \begin{pmatrix}
1 & {x}^\top \\ {x} & {X}
\end{pmatrix} \succeq \mathbf{0}.
\end{aligned}
\end{align}
Here we include the squared representation of the linear constraint, see also~\eqref{constr:quadratic}.
Note that the SDP relaxation~\eqref{ternaryProblemSDPbasic} does not include the constraint  $x^\top  \mathbf{1}_n = 0$, as it is redundant in the presence of the constraint $\langle {\mathbf J},X\rangle = 0$.  
\begin{lemma} \label{Lem:nullspace}
Let  $(X,x)$ satisfy $X-xx^\top \succeq \mathbf{0}$
and $\langle \mathbf{J}, X \rangle  = 0$. Then,   $x^\top  \mathbf{1}_n = 0$.
\end{lemma}
\begin{proof}
   Let $Z:=X-xx^\top$. Then,
   $
   0 = \mathbf{1}_n^\top X \mathbf{1}_n = 
   \mathbf{1}_n^\top (Z + xx^\top ) \mathbf{1}_n$ $\Rightarrow$ 
  $(x^\top \mathbf{1}_n)^2 = - \mathbf{1}_n^\top Z \mathbf{1}_n \leq 0,
   $
   since $Z\succeq \mathbf{0}.$ Therefore, $x^\top \mathbf{1}_n = 0.$
\end{proof}
One can  verify that $v=(0, \mathbf{1}_n^\top)^\top$
is an eigenvector corresponding to the zero eigenvalue of any matrix that is feasible for \eqref{ternaryProblemSDPbasic}. Therefore, the SDP relaxation \eqref{ternaryProblemSDPbasic} has no Slater feasible points. We apply facial reduction to obtain an equivalent SDP relaxation that is Slater feasible. For an introduction to facial reduction, we refer the interested reader to~\cite{drusvyatskiy2017many} and the references therein. 

Let $\mathcal{R} := \left\{ x \in \mathbb{R}^{n+1} \, : \, \, v^\top x = 0 \right\}$ denote the orthogonal complement of $v$ and let
\begin{align}
    F_\mathcal{R} := \left \{ Y \in \mathcal{S}^{n+1}_+ \, : \, \, \Col(Y) \subseteq \mathcal{R} \right\}. 
\end{align}
It follows from Lemma~\ref{Lem:nullspace} that the feasible set of~\eqref{ternaryProblemSDPbasic} is contained in $F_\mathcal{R}$. It is well-known that sets of the form $F_\mathcal{R}$ are faces of $\mathcal{S}^{n+1}_+$~\cite{drusvyatskiy2017many}.
Now, $F_\mathcal{R}$ can be expressed as
$ F_\mathcal{R} = W\mathcal{S}^n_+ W^\top$,
where $W \in \mathbb{R}^{(n+1)\times n}$ is such that its columns form a basis of $\mathcal{R}$. A possible and sparse choice for $W$ is the matrix
\begin{align} \label{def:W}
    W = 
    \begin{pmatrix}
        1 & 0 & \mathbf{0}_{n-1}^\top  \\
        \mathbf{0}_{n-1} & \mathbf{1}_{n-1}  & -\mathbf{I}_{n-1}
    \end{pmatrix}^\top. 
\end{align}
We rewrite~\eqref{ternaryProblemSDPbasic} by replacing $\begin{pmatrix}
    1 & x^\top \\ x & X
\end{pmatrix}$ by $WZW^\top$, where $Z\in \mathcal{S}^n_+$. This yields:
\begin{align} \label{ternaryProblemSlaterFeas}
\begin{aligned} 
\min ~~ &  \left \langle \begin{pmatrix}
    0 & \frac{1}{2} c^\top \\ 
  \frac{1}{2}  c & Q
\end{pmatrix}, WZW^\top \right\rangle  \\
\st ~~ 
& \diag(WZW^\top) \geq WZW^\top {\color{black} \mathbf{e}_0}, \quad 
 \diag(WZW^\top) \geq - WZW^\top {\color{black} \mathbf{e}_0} \\
& \diag(WZW^\top) \leq {\color{black}{\mathbf 1}_{n+1}}, \quad {\color{black} \mathbf{e}_0^\top} WZW^\top {\color{black} \mathbf{e}_0} = 1, \quad
 Z \succeq \mathbf{0},
\end{aligned}
\end{align}
which is equivalent to~\eqref{ternaryProblemSDPbasic}. Observe that all constraints that are present in~\eqref{ternaryProblemSDPbasic} are explicitly present in~\eqref{ternaryProblemSlaterFeas}, except for the constraint $\langle \mathbf{J}_n, X \rangle = 0$. This constraint, which after substitution is equivalent to $\left \langle \mathbf{J}_n, \begin{pmatrix}
    \mathbf{0}_n & \mathbf{I}_n
\end{pmatrix}  WZW^\top \begin{pmatrix}
    \mathbf{0}_n & \mathbf{I}_n.
 \end{pmatrix}^\top \right \rangle = 0$, becomes redundant in~\eqref{ternaryProblemSlaterFeas} due to the construction of $W$. 
We now state the following result regarding~\eqref{ternaryProblemSlaterFeas}. 

\begin{lemma}
    The SDP~\eqref{ternaryProblemSlaterFeas} is strictly feasible and equivalent to~\eqref{ternaryProblemSDPbasic}. 
\end{lemma}
\begin{proof}
    The equivalence between the two programs follows from the construction of~\eqref{ternaryProblemSlaterFeas} and the discussion above. To show that~\eqref{ternaryProblemSlaterFeas} is strictly feasible, let 
    \begin{align*}
        \hat{Z} := \frac{1}{n}\mathbf{e}_1\mathbf{e}_1^\top + \frac{1}{n} \sum_{k = 2}^n (\mathbf{e}_1 + \mathbf{e}_k)(\mathbf{e}_1 + \mathbf{e}_k)^\top = 
        \begin{pmatrix}
            1 & \frac{1}{n}\mathbf{1}_{n-1}^\top  \\[0.8em]
            \frac{1}{n}\mathbf{1}_{n-1} & \frac{1}{n}\mathbf{I}_{n-1}
        \end{pmatrix}. 
    \end{align*}
    Using $W$ as in~\eqref{def:W}, we define $\hat{Y} := W\hat{Z}W^\top \in \mathcal{S}^{n+1}$.
    By direct verification, it follows that $\diag(\hat{Y}) \geq \hat{Y}{\color{black}\mathbf{e}_0}$, $\diag(\hat{Y}) \geq -\hat{Y}{\color{black}\mathbf{e}_0}$, $\diag(\hat{Y}) \leq \mathbf{1}_{n+1}$ and ${\color{black}\mathbf{e}_0}1^\top \hat{Y}{\color{black}\mathbf{e}_0} = 1$. Moreover, since $\hat{Z}$ is defined as a convex combination of rank-one PSD matrices, we know $\hat{Z} \succeq \mathbf{0}$. Thus, $\hat{Z}$ is feasible for~\eqref{ternaryProblemSlaterFeas}. To show that $\hat{Z}$ is contained in the interior of $\mathcal{S}^n_+$, observe that  $M := \begin{pmatrix}
        \mathbf{e}_1 & \mathbf{e}_1 + \mathbf{e}_2 & \dots &\mathbf{e}_1 + \mathbf{e}_n
    \end{pmatrix}$ has rank $n$. Since $\hat{Z} = \frac{1}{n} MM^\top$, it follows that $\hat{Z}$ has full rank, implying that $\hat{Z} \succ \mathbf{0}$.     
\end{proof}

\subsection{Ternary quadratic problem with a ratio objective}
We study the quadratic program with a ratio objective and ternary variables, which we refer to as TQP-Ratio. 
We propose two approaches for solving the  TQP-Ratio. The first approach is based on a version of Dinkelbach’s algorithm for fractional programming~\cite{Dinkelbach1967NonlinearFractional}
 that relies on a B\&B algorithm for the QUTO, see Section \ref{sect:QUTO}. 
The second approach exploits the convexification technique  from~\cite{he2025convexification} to derive an SDP relaxation for the TQP-Ratio, that is used within our B\&B algorithm for solving instances of the TQP-Ratio to optimality.

We consider the following fractional quadratic optimization problem: 
\begin{align} \label{eq:frac_ternary}
\min_{x\in\{0,\pm 1\}^n}\ 
\frac{f(x)}{g(x)},
\end{align}
where 
\begin{align}\label{def:fg}
 f(x):=x^\top A x + a^\top x + a_0, \quad
 g(x):=x^\top B x + b^\top x + b_0 > 0, 
\end{align}
for all $x\in\{0, \pm 1\}^n.$
Ternary quadratic problems  with a ratio objective appear in many applications, see e.g., \cite{hotho2006information,Trevisan2009MaxCut}.
In the sequel, we provide two approaches for solving~\eqref{eq:frac_ternary} to optimality. \\


\noindent {\bf Parametric optimization.} 
Dinkelbach~\cite{Dinkelbach1967NonlinearFractional} introduced a framework for solving convex fractional programs. Dinkelbach's approach was later extended to solve general fractional programs~\cite{Rodenas1999}. To the best of our knowledge, we are the first to exploit Dinkelbach's  approach for ternary quadratic problems. To this end, we introduce the following parameterized optimization problem  w.r.t.~$\lambda \in \mathbb{R}$  associated with \eqref{eq:frac_ternary}: 
\begin{align}\label{Phi}
\Phi(\lambda ) := \min_{x \in \{0, \pm 1\}^n} \left (  f(x) -  \lambda g(x)  \right ).   
\end{align} 
It is straightforward to verify that the function $\Phi(\cdot)$ is concave and strictly monotonically decreasing.
Moreover, for any $\lambda = f(x)/g(x)$ with $x \in \{0, \pm 1\}^n$, it holds that   $\Phi(\lambda) \leq 0$. 
As shown by Dinkelbach~\cite{Dinkelbach1967NonlinearFractional}, we have
\begin{align*}
      \lambda^* = \min_{x \in \{0,\pm 1\}^n} \frac{f(x)}{g(x)} \quad \Longleftrightarrow \quad \Phi(\lambda^*) = 0,
\end{align*}
where the minimum in $\Phi(\lambda^*)$ is attained at $x^* \in  \argmin_{x \in \{0,\pm 1\}^n} {f(x)}/{g(x)}$. 
In other words, to solve the problem~\eqref{eq:frac_ternary}, it suffices to find the unique root of the parametrized function $\Phi(\cdot)$. Due to the properties of the function $\Phi(\cdot)$, this root can be found by iteratively evaluating the function $\Phi(\lambda)$ at an incumbent point $\lambda = f(x)/g(x)$ for some $x \in \{0,\pm 1\}^n$, which is updated accordingly if $\Phi(\lambda) < 0$.

For a given  $\lambda$, each main iteration of Dinkelbach’s algorithm \cite{Rodenas1999} solves the parametrized ternary optimization problem \eqref{Phi}, that can be equivalently formulated as follows:
\begin{align} \label{DinkelSubproblem}
\begin{aligned}
\Phi(\lambda )  = \min ~~& 
\left \langle  \begin{pmatrix} 
a_0- \lambda b_0 & \frac{1}{2} (a^\top-\lambda b^\top) \\
\frac{1}{2}(a-\lambda b ) &   A - \lambda B
 \end{pmatrix} , \begin{pmatrix} 
1 & {x}^\top \\ {x} & {X}
\end{pmatrix}   \right \rangle \\
\st ~~&   \begin{pmatrix} 
1 & {x}^\top \\ {x} & {X}
\end{pmatrix} \succeq \mathbf{0}, \quad \diag(X) = |x|, \quad  X \in \{0, \pm 1 \}^{ n \times n }.
\end{aligned}
\end{align}
The optimization problem \eqref{DinkelSubproblem} is the QUTO  whose basic SDP relaxation is~\eqref{QUTOSDPbasic}.
Schaible~\cite{Schaible} proved superlinear convergence of Dinkelbach’s algorithm.
A similar convergence result can be established for a version of Dinkelbach’s algorithm applied to the TQP-Ratio problem.
Our version of Dinkelbach’s algorithm iteratively updates the parameter $\lambda$ by solving a sequence of parametric  subproblems~\eqref{DinkelSubproblem}, see Algorithm \ref{alg:Dinkelbach}. Observe that the starting solution $x$ is computed using a Variable Neighborhood Search heuristic, which is introduced in Section~\ref{sect:VNS}. \\

\begin{algorithm}[!ht]
\footnotesize
\caption{Dinkelbach’s algorithm for the TQP-Ratio} \label{alg:Dinkelbach}
\begin{algorithmic}[1]
\STATE  \textbf{input}  $(A,a,a_0)$, $(B,b,b_0)$, $\epsilon >0$
\STATE $x \gets$ VNS-Heuristic (Section \ref{sect:VNS}), \,
 $\lambda \gets f(x)/g(x) $
  \WHILE{true}
    \STATE $x \in \argmin_{x\in \{0,\pm1\}^n} \left (f(x)-\lambda g(x) \right ) $ \\[0.7ex]
    \IF{$| f(x)- \lambda g(x)| < \epsilon$}
    \STATE \textbf{return} $x^* \gets x$, $\lambda^* \gets f(x^*)/g(x^*)$
        \ENDIF
        \STATE $\lambda \gets f(x)/g(x) $
  \ENDWHILE
\STATE \textbf{return} $x^*$, $\lambda^*$
\end{algorithmic}
\end{algorithm}

\noindent
{\bf An SDP relaxation for TQP-Ratio with no fractional terms.}
We follow the approach {\color{black} from~\cite{he2025convexification} and  reformulate  \eqref{eq:frac_ternary} as an  optimization problem without fractional terms.}
Define the rescaling
\begin{align} \label{rhoRescale}
\rho := g(x)^{-1},\qquad
y := \rho x,\qquad
Y := \rho\, x x^\top .
\end{align}
Then for any $x\in\{0, \pm 1\}^n$ we have the identities
${f(x)}/{g(x)} \ =\ \langle A,Y\rangle + a^\top y + a_0\rho$ and
$\langle B,Y\rangle + b^\top y + b_0\rho \ =\ 1.$ {\color{black} These identities are used to derive the following convexification of~\eqref{eq:frac_ternary}.}
{\color{black} 
\begin{proposition}
    The fractional quadratic optimization problem~\eqref{eq:frac_ternary} 
    attains the same optimal value as 
    \begin{align} \label{eq:exact_proj_conv}
    \begin{aligned}
\min_{\rho,y,Y}\quad 
& \langle A,Y\rangle + a^\top y + a_0 \rho\\
\text{s.t.}\quad
& \langle B,Y\rangle + b^\top y + b_0 \rho = 1,\\
& \rho\ge 0,\,\, (y,Y) \in \rho \, {\rm conv} \left ( \widehat{\mathcal F}_n^1\right ), 
\end{aligned}
\end{align}
where $\widehat{\mathcal F}_n^1 := \left \{(x,xx^\top ) \, : \, \, x \in \{0, \pm 1\}^n \right\}.$
\end{proposition}
\begin{proof}
Let  $x\in\{0,\pm1\}^n$  and define   $\rho$, $y$ and $Y$ as in~\eqref{rhoRescale}. Then, it readily follows that  $(\rho,y,Y)$ are feasible for~\eqref{eq:exact_proj_conv}, and that the corresponding objectives are equal.
If $\widehat{v}$ denotes the optimal value of $\eqref{eq:exact_proj_conv}$, and 
$v^\star$  the optimal value of \eqref{eq:frac_ternary},
it follows that $\widehat{v} \leq v^\star$. 

Conversely, let $(\rho,y,Y)$  be feasible  for~\eqref{eq:exact_proj_conv}. 
If  $\rho =0$ then
$(y,Y) \in 0 \, {\rm conv} \left ( \widehat{\mathcal F}_n^1\right ) = \left \{ (0,0)\right \},$ 
which results in an infeasible point for~\eqref{eq:exact_proj_conv}
due to $\langle B,Y\rangle + b^\top y + b_0 \rho = 1.$ Therefore, we have that a feasible $\rho$ for \eqref{eq:exact_proj_conv} satisfies $\rho >0$.

Since $(y,Y) \in \rho \, {\rm conv} \left ( \widehat{\mathcal F}_n^1\right )$, there exist $x_i \in \{ 0, \pm 1\}^n$ and $\lambda_i \in [0,1]$ for $i\in [r]$ such that 
$\sum_{i=1}^r \lambda_i =1$,
$ y = \rho\sum_{i=1}^r \lambda_i x_i $ and 
$Y = \rho\sum_{i=1}^r \lambda_i x_i x_i^\top.$
From  this,  \eqref{def:fg} and~the equality constraints in \eqref{eq:exact_proj_conv}, we get 
$$
1 = \langle B,Y\rangle + b^\top y + b_0 \rho =
 \rho\sum_{i=1}^r \lambda_i \left ( 
 x_i^\top  B x_i + b^\top x_i + b_0 
 \right )  =  \rho\sum_{i=1}^r \lambda_i g(x_i).
$$
Now, define $\widehat{\lambda}_i :=  \rho \lambda_i g(x_i)$ for all $i \in [r]$.
Since $g(x_i) >0$, we have $\widehat{\lambda}_i \geq 0 $. From this  and
$\sum_{i=1}^{r} \widehat{\lambda}_i=1$ it follows that 
$\widehat{\lambda}_i \leq 1.$
Now,
$$
\langle A,Y\rangle + a^\top y + a_0 \rho = 
\rho \sum_{i=1}^r \lambda_i f(x_i) 
= \sum_{i=1}^r \widehat{\lambda}_i \frac{f(x_i) }{g(x_i)}.
$$
Since $f(x_i)/g(x_i) \geq v^\star$ for all $i \in [r]$, we have
\begin{align*}
\langle A,Y\rangle + a^\top y + a_0 \rho =    \sum_{i = 1}^r \widehat{\lambda}_i \frac{f(x_i)}{g(x_i)} \geq \sum_{i=1}^r \widehat{\lambda}_i v^\star = v^\star.
\end{align*}
As this holds for every feasible $(\rho, y, Y)$ of \eqref{eq:exact_proj_conv}, we can conclude that $\widehat{v} \geq v^\star$. Combining the latter with $\widehat{v} \leq v^\star$ we obtain  $v^\star = \widehat{v}$. Hence \eqref{eq:frac_ternary} and \eqref{eq:exact_proj_conv} have the same optimal value.
\end{proof}
}

To derive an SDP relaxation for \eqref{eq:exact_proj_conv}, we introduce the positive semidefinite matrix variable
$Z \ :=\
\begin{pmatrix}
\rho & y^\top\\
y & Y
\end{pmatrix} \succeq \mathbf{0}.$
Moreover, for $x_i\in\{0, \pm 1\}$ we have the implications
$Y_{ii}=\rho x_i^2\in\{0,\rho\}$ and 
$|y_i|= \rho x_i^2=Y_{ii}$.
This motivates the linear inequalities
$0\le Y_{ii}\le \rho$ and $-Y_{ii}\le y_i\le Y_{ii}$ for all $i\in [n].$
The previous discussion 
yields the following SDP relaxation of~\eqref{eq:exact_proj_conv}:
\begin{equation}
\label{eq:sdp_relax}
\begin{aligned}
\min_{}\quad 
& \langle A,Y\rangle + a^\top y  + a_0\rho  \\
\text{s.t.}\quad
& \langle B,Y\rangle + b^\top y  + b_0\rho = 1 \\
& \diag(Y) \geq y, \quad \diag(Y) \geq - y \\
&  \rho\ge 0, \quad \diag(Y) \leq \rho {\mathbf 1}_n,
\quad  
\begin{pmatrix}
\rho & y^\top\\
y & Y
\end{pmatrix}\succeq \mathbf{0}. 
\end{aligned}
\end{equation}
In order to strengthen the above SDP relaxation, one can add scaled versions of the cuts discussed 
in \Cref{Sect:SDP and cuts}. Rescaling is performed with respect to~$\rho$, see \eqref{rhoRescale}.
Our preliminary results suggest that adding the following constraints to the SDP relaxation~\eqref{eq:sdp_relax}
is most beneficial for strengthening it:
Scaled  triangle inequalities
\begin{align} \begin{aligned}
     Y_{ij} +  Y_{ik} + Y_{jk}  &\geq -\rho, \,\, &
     -Y_{ij} + Y_{ik} - Y_{jk}  &\geq -\rho, \,\, \\
   Y_{ij} - Y_{ik} - Y_{jk}     &\geq -\rho, \,\, &
     -Y_{ij} - Y_{ik} + Y_{jk}  & \geq -\rho,
     \end{aligned}
     \end{align}
for    $1\leq i < j < k \leq n$,  see also \eqref{C:eq1}--\eqref{C:eq2.3};
scaled RLT inequalities 
\begin{align}  
\begin{aligned}
Y_{ij} + y_i+ y_j &\geq -\rho, \,\, &
Y_{ij} - y_i- y_j &\geq -\rho, \,\, \\
-Y_{ij} + y_i- y_j &\geq -\rho, \,\, &
-Y_{ij} - y_i + y_j &\geq  -\rho,
\end{aligned}
\end{align}
for $i,j\in [n]$, $i<j$,  see also  \eqref{rlt:1}--\eqref{rlt:4};   split inequalities of the form \eqref{newsplit1}--\eqref{newsplit4}, and non-standard split inequalities of the form \eqref{C:eq4}.
Split inequalities for the SDP relaxation~\eqref{eq:sdp_relax} are obtained from the inequalities~\eqref{newsplit1}--\eqref{newsplit4} by replacing $X$ by $Y$, and $x$ by $y$, and similar for non-standard split inequalities.

\section{SDP-based branch-and-bound algorithm for the TQP} \label{Sec:BandB}
In this section we incorporate our basic SDP relaxation and the derived cutting planes from Section~\ref{Sect:SDP and cuts} in a specialized branch-and-bound algorithm to solve TQP instances. We tailor this algorithm to the three special cases discussed in Section~\ref{Sec:SpecialCases}. We refer to the resulting method as \texttt{SDP-B\&B}.

\texttt{SDP-B\&B} systematically explores a tree of subproblems, starting from the root node. At each node of the tree, a lower bound is obtained by strengthening the basic SDP relaxation~\eqref{ternaryProblemSDP} by a cutting-plane procedure. For the special cases of the  QUTO and the TQP-Linear, this relaxation boils down to solving~\eqref{QUTOSDPbasic} and \eqref{ternaryProblemSDPbasic}, respectively. 
A primal heuristic is exploited to find valid upper bounds for the TQP. This heuristic takes a solution of the strengthened SDP relaxation as an input, hence ideally resulting in better solutions when the lower bounds improve.
Our branching strategy takes a fractional variable and fixes this variable to one of the ternary values in $\{0, \pm 1\}$, so that the parent node (i.e., subproblem) is split into three child nodes. Details on the ingredients of our algorithm are provided in the next subsections. 

\subsection{Primal heuristic}
\label{sect:VNS} 
We design a Variable Neighborhood Search (VNS) he\-uristic to compute high-quality feasible solutions.
The method follows the classical VNS framework \cite{mladenovic1997variable}, alternating between (i) a randomized shaking phase to diversify the search and escape local minima and (ii) a deterministic best-improvement local search phase
to intensify around promising solutions \cite{ hansen2010variable}.
We first present the heuristic for the QUTO, then describe the adaptation for TQP-Linear, and finally detail the specialization for the TQP-Ratio.

\paragraph{VNS for QUTO}
For a current solution $x$, we maintain the auxiliary vector $\alpha(x)=Qx$. The local neighborhood is defined by all single-coordinate ternary changes
$x_i \leftarrow v$, with $v\in\{-1,0,1\}\setminus\{x_i\}$.
Let $\Delta=v-x_i\in\{-2,-1,1,2\}$, $\tilde{x}$ denote the resulting solution, and $f(x) = x^\top Q x + c^\top x$ denote the objective of the QUTO, then the objective variation is
$$
f({\tilde{x}})
=
f(x)
+2\Delta\,\alpha_i
+\Delta^2 Q_{ii}
+\Delta\,c_i .
$$
Hence, any trial move can be evaluated in $O(1)$ time.
If a move is accepted, the auxiliary vector is updated as $\alpha \leftarrow \alpha + \Delta\,Q_{:i}$, where $Q_{:i}$ is the $i$th column of $Q$,
which requires $O(n)$ time.
Algorithm~\ref{alg:vns_quto} summarizes the VNS scheme.
The local search phase evaluates all $2n$ candidate moves in $O(n)$ time, selects the best improving one,
and iterates until no improving move exists, yielding a 1-move local optimum.
The shaking phase randomly modifies $s$ coordinates with $s\in [s_\textrm{min},s_\textrm{max}]$.

Whenever an improvement over the incumbent solution is found, the neighborhood radius is
reset to its minimum value; otherwise, it is increased to promote diversification.

\begin{algorithm}[!ht]
\footnotesize
\caption{Variable Neighborhood Search for QUTO}\label{alg:vns_quto}
\begin{algorithmic}[1]
\STATE  \textbf{input}  $s_\textrm{min}$, $s_\textrm{max}$, $s_\textrm{step}$, ${iter}_\textrm{max}$
\STATE Choose $x\in\{0,\pm 1\}^n$ and compute $(f, \alpha)$
\STATE $(x_{\text{best}},f_{\text{best}}) \gets (x,f)$
\FOR{$it=1$ to $iter_{\max}$}
  \STATE $s \gets s_{\min}$
  \WHILE{$s \le s_{\max}$}
    \STATE $x \gets$ Shake$(x,s)$
    \STATE $x \gets$ LocalSearch$(x)$
    \STATE Recompute $f$
    \IF{$f < f_{\text{best}}$}
      \STATE $(x_{\text{best}},f_{\text{best}}) \gets (x,f)$
      \STATE $s \gets s_{\min}$
    \ELSE
      \STATE $x \gets x_{\text{best}}$ and recompute $(f,\alpha)$
      \STATE $s \gets s + s_{\text{step}}$
    \ENDIF
  \ENDWHILE
\ENDFOR
\STATE \textbf{return} $x_{\text{best}}$
\end{algorithmic}
\end{algorithm}

\paragraph{VNS for TQP-Linear}
Here, feasibility is restricted by the linear constraint $\mathbf{1}^\top_n x =0$.
The local neighborhood is therefore defined by paired ternary changes on two distinct
indices $i\neq j$, $x_i\leftarrow v_i, \ x_j\leftarrow v_j$, with $v_i,v_j\in\{-1,0,1\}$ such that $(v_i-x_i)+(v_j-x_j)=0$. Let $\Delta_i=v_i-x_i$ and $\Delta_j=v_j-x_j$.
Using again $\alpha(x)=Qx$, the objective variation is
\[
f({\tilde{x}})
= f(x)
+2\Delta_i\,\alpha_i
+2\Delta_j\,\alpha_j
+\Delta_i^2 Q_{ii}
+\Delta_j^2 Q_{jj}
+2\Delta_i\Delta_j Q_{ij}
+\Delta_i c_i
+\Delta_j c_j .
\]
Each feasible paired move can be evaluated in $O(1)$ time.
When such a move is accepted, the auxiliary vector is updated as $\alpha \leftarrow \alpha + \Delta_i Q_{:i} + \Delta_j Q_{:j}$.
The local search phase explores this 2-move neighborhood and applies the best
improving feasible move until convergence.
The shaking phase is adapted accordingly by generating random feasible paired moves.
Apart from the use of paired moves, the overall VNS structure is identical to the QUTO case. Since the balanced neighborhood contains $O(n^2)$ candidate moves, one complete
neighborhood scan requires $O(n^2)$ time.

\paragraph{VNS for TQP-Ratio}
For the fractional case, the local neighborhood
again consists of single-coordinate ternary moves.
The objective is the ratio denoted by $R(x)=f(x)/g(x)$.
For a current solution $x$, we maintain the auxiliary vectors $\alpha^A(x)=Ax$, $\alpha^B(x)=Bx$. For a move $x_i\leftarrow v$ with $\Delta=v-x_i$, the numerator and denominator variations are
\begin{align*}
\begin{aligned}
f(\tilde{x}) = f(x) + 2\Delta\,\alpha^A_i + \Delta^2 A_{ii} + \Delta\,a_i, \qquad
g(\tilde{x}) = g(x) + 2\Delta\,\alpha^B_i + \Delta^2 B_{ii} + \Delta\,b_i .
\end{aligned}
\end{align*}
Thus, each trial move and the corresponding ratio value can be evaluated in $O(1)$ time.
If the move is accepted, the auxiliary vectors are updated in $O(n)$ time as $\alpha^A \leftarrow \alpha^A + \Delta\,A_{:i}$, $
\alpha^B \leftarrow \alpha^B + \Delta\,B_{:i}.$

The local search phase selects the best feasible move that strictly improves the ratio and
terminates when no improving move exists, yielding a solution that is locally optimal with
respect to all single-coordinate ternary changes.
The VNS structure (shaking, intensification, and diversification) is identical to the
QUTO case. 


\subsection{Cut separation}
\label{ref: cut separation}
At each node in the B\&B tree, violated triangle, split, RLT, and \(k\)-gonal inequalities are iteratively added to the relaxation.  The separation problems for triangle, (non-standard) split, and RLT inequalities can be solved exactly by exhaustive enumeration. In contrast, the number of $k$-gonal inequalities grows rapidly with the problem dimension, making exhaustive enumeration computationally prohibitive. Consequently, we employ heuristic separation strategies for these classes. Following the approach proposed in \cite{gusmeroli2022biqbin,HrgaPovh2021} for the max-cut problem, the separation of \(k\)-gonal inequalities is formulated as the quadratic assignment problem (QAP)
$\min_{P \in \Pi} \ \langle H, P X^\star P^\top \rangle,
$
where \(\Pi\) denotes the set of permutation matrices, \(X^\star\) is the current SDP solution, and 
\(H = vv^\top\) encodes the structure of the corresponding inequality defined by a vector \(v\); 
see~\eqref{oddS}. The vector \(v\) has nonzero entries in its first \(k\) components, with 
\(k \in \{5, 7, 9\}\) determining the inequality type. 
Since the QAP is $\mathcal{NP}$-hard, we solve it approximately using simulated annealing combined with a 
multistart strategy based on random initial permutations. 


\subsection{Branching strategy}
\label{sect: branching strategy}
We employ a ternary branching scheme reflecting the discrete domain \(x_i \in \{0,\pm1\}\).
Whenever the SDP solution at a node is not integral, branching is performed by fixing one selected 
variable to each of its three possible values. In particular, branching on variable \(x_i\) 
creates three child nodes by imposing $(x_i, X_{ii}) \in \{(-1,1),\; (0,0),\; (1,1)\}$.
The branching variable is selected using a most-fractional rule adapted to the ternary domain. 
For each variable \(x_i\), its fractionality is defined as the distance to the nearest feasible 
value, $\phi(x_i) = \min\bigl\{ |x_i + 1|,\; |x_i|,\; |x_i - 1| \bigr\}$,
and we choose $i^\star \in \arg\max_i \; \phi(x_i)$.
For the TQP-Ratio formulation, we instead branch on the variable exhibiting the largest violation 
of the scaled rank-one constraint \(\rho Y_{jk} \approx y_j y_k\). Given the SDP solution 
\((\rho^\star,y^\star,Y^\star)\), we select
\begin{align*}
    i^\star \in \arg\max_{j\in[n]}
\sum_{k=1}^{n} \textcolor{black}{
\left|
Q_{kj}\left(y_j^\star y_k^\star-\rho^\star Y_{jk}^\star\right)
\right|}.
\end{align*}
The scaled rank-one constraints follow from 
{\color{black}$(y,Y) \in \rho \, {\rm conv} \left ( \widehat{\mathcal F}_n^1\right )$}.
Empirically, this branching rule outperforms the most-fractional strategy for solving 
\eqref{eq:exact_proj_conv}, although detailed computational results are omitted for brevity.

\subsection{Implementation details}\label{sec:impl}

Before starting the B\&B procedure, we perform 100 local searches  using the VNS algorithm initialized from randomly generated points, and the best solution found is used to initialize the global upper bound. In the VNS heuristic, we set $s_\textrm{min}=2$, $s_\textrm{max}=n$, $s_\textrm{step}=2$, and $iter_\textrm{max}=3$. Nodes of the B\&B tree are explored using a best-first search strategy. At each cutting-plane iteration, inequalities are tested for violation with tolerance \(10^{-3}\), and at most 5{,}000 of the most violated inequalities are added. 
For the $k$-gonal inequalities, the simulated annealing procedure is executed 500 times for pentagonal ($k=5$) inequalities and 1{,}000 times for heptagonal ($k=7$) and enneagonal ($k=9$) inequalities. The heuristic is run for all the distinct vectors $v$ in \eqref{oddS}. The cutting-plane procedure is terminated when fewer than \(n\) violated inequalities are identified. The algorithm terminates when the relative optimality gap satisfies $(UB - LB)/\textcolor{black}{|UB|} \leq 10^{-4}$, where $UB$ and $LB$ denote the best upper and lower bounds, respectively. 
{\tt SDP-B\&B} is implemented in {\tt MATLAB} (version 2025a). 
The SDP relaxation at each node is solved with the interior-point solver 
{\tt MOSEK} (version 11.0). {\textcolor{black}{Since interior-point SDP solvers generally do not benefit from a warm-start solution, the SDP relaxations are solved without warm-start information each time new cuts are added.}} 
All experiments are performed on a macOS system equipped with an Apple M4 Max processor (14 cores) and 36~GB of RAM, running macOS version~15.5. \textcolor{black}{The source code is publicly available at \url{https://github.com/antoniosudoso/tqp-sdp}.}

\section{Numerical results} \label{sect:numResults}
In this section, we report the computational results for the QUTO, TQP-Linear, and TQP-Ratio formulations of the proposed \texttt{SDP-B\&B} algorithm and compare it with 
state-of-the-art solvers.

We conducted experiments at the {\color{black}root node} to evaluate the effectiveness of different classes of valid inequalities introduced in Section \ref{Sect:SDP and cuts}. Based on these experiments that we do not report for the sake of brevity, we implement within \texttt{SDP-B\&B} a cutting-plane strategy that first enumerates all computationally inexpensive classes of inequalities jointly, namely triangle, RLT, and split inequalities, and subsequently generates violated $ k$-gonal inequalities with $k=5$ and $k=7$. Observe that the enneagonal (i.e., $k$-gonal with $k = 9$) inequalities are no longer considered, due to their high computational cost.

\subsection{\textcolor{black}{Branch-and-bound results on randomly generated instances}}
In this section, we evaluate the computational performance of {\tt SDP-B\&B} \textcolor{black}{on randomly generated instances}.
\textcolor{black}{We consider three types of instances for the  QUTO and TQP-Linear called respectively Type-1, Type-2 and Type-3. Furthermore, we generate a set of random instances for the TQP-Ratio. Details on instance generation and parameter configuration are described in Appendix \ref{app:instance_generation}.
For all problem classes, we consider problem sizes $n \in \{60,70,80,90,100,110,120\}$ and three random seeds for each parameter configuration, resulting in 189 instances for each of the QUTO and TQP-Linear, and 63 instances for the TQP-Ratio.
}

We compare {\tt SDP-B\&B} with the commercial solver {\tt GUROBI} \cite{gurobi} (version 12.0.2). We set the optimality gap tolerance to $10^{-4}$, as for {\tt SDP-B\&B}. \textcolor{black}{Moreover, to ensure a fair comparison, {\tt GUROBI} is provided with the same incumbent solution obtained from the VNS runs used to initialize {\tt SDP-B\&B}, through
a MIP start.} We also conducted experiments using {\tt BARON} (version 25.3.19) \cite{Sahinidis2016BARON} and {\tt SCIP} \cite{Achterberg2009} (version 10.0.0); however, both solvers consistently exhibited inferior performance compared to {\tt GUROBI}. For this reason, their results are not reported in the following analysis. The comparison focuses on the QUTO, TQP-Linear, and TQP-Ratio instances, with problem sizes $n \in \{60, 70, 80, 90\}$. For these instances, we set a time limit of 3600 seconds. In addition, to assess the scalability of the proposed approach, we also consider large-scale instances with $n \in \{100, 110, 120\}$. For these cases, we run only {\tt SDP-B\&B} with an increased time limit of 10,800 seconds. This choice is motivated by the observation that {\tt GUROBI} already reached the time limit of 3600 seconds on a substantial number of instances in the smaller size range ($n$ between 60 and 90).

\begin{figure}[!ht]
    \centering
    \includegraphics[scale=0.21]{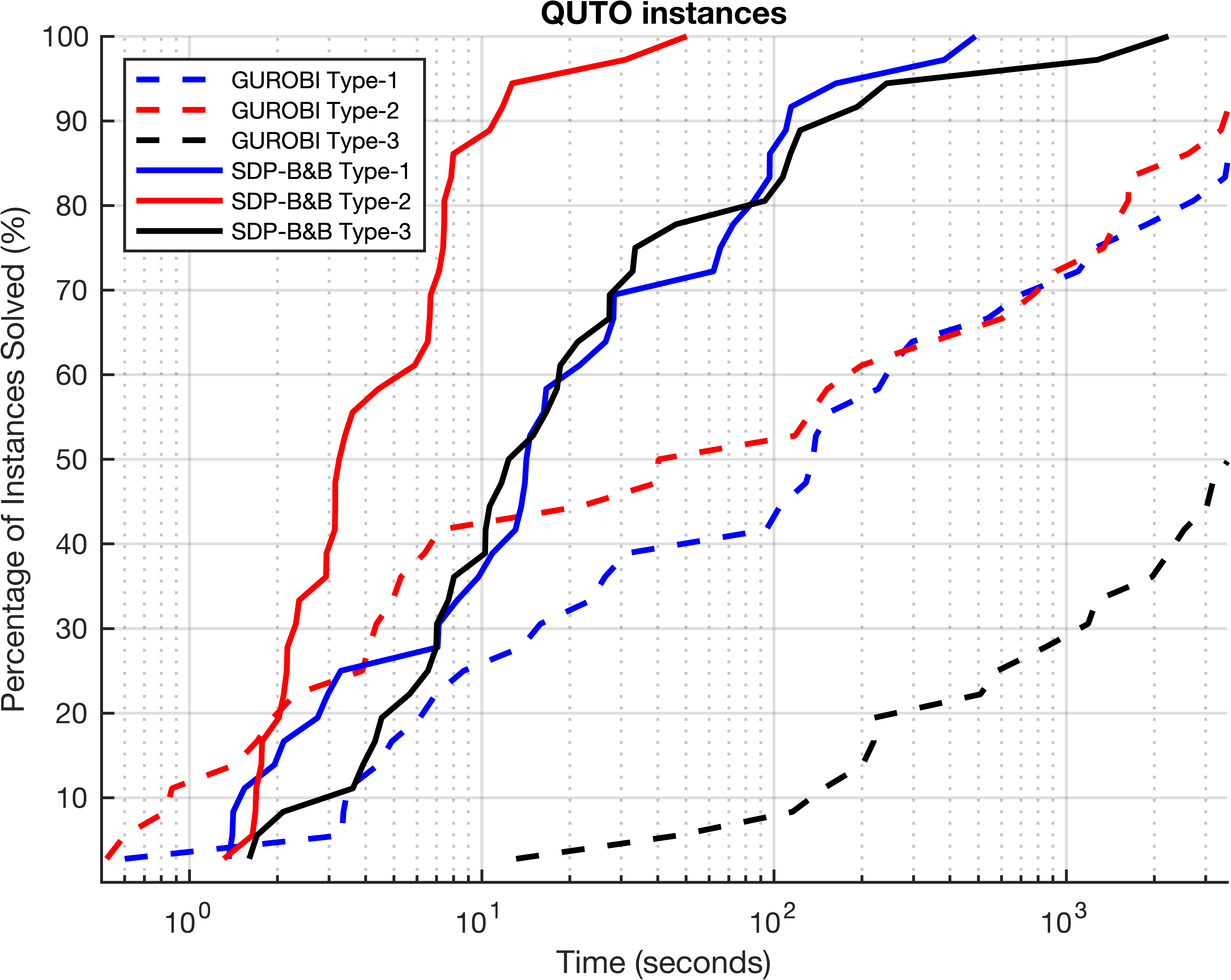}%
    \hspace{2.5mm}
    \vspace{2mm}
    \includegraphics[scale=0.21]{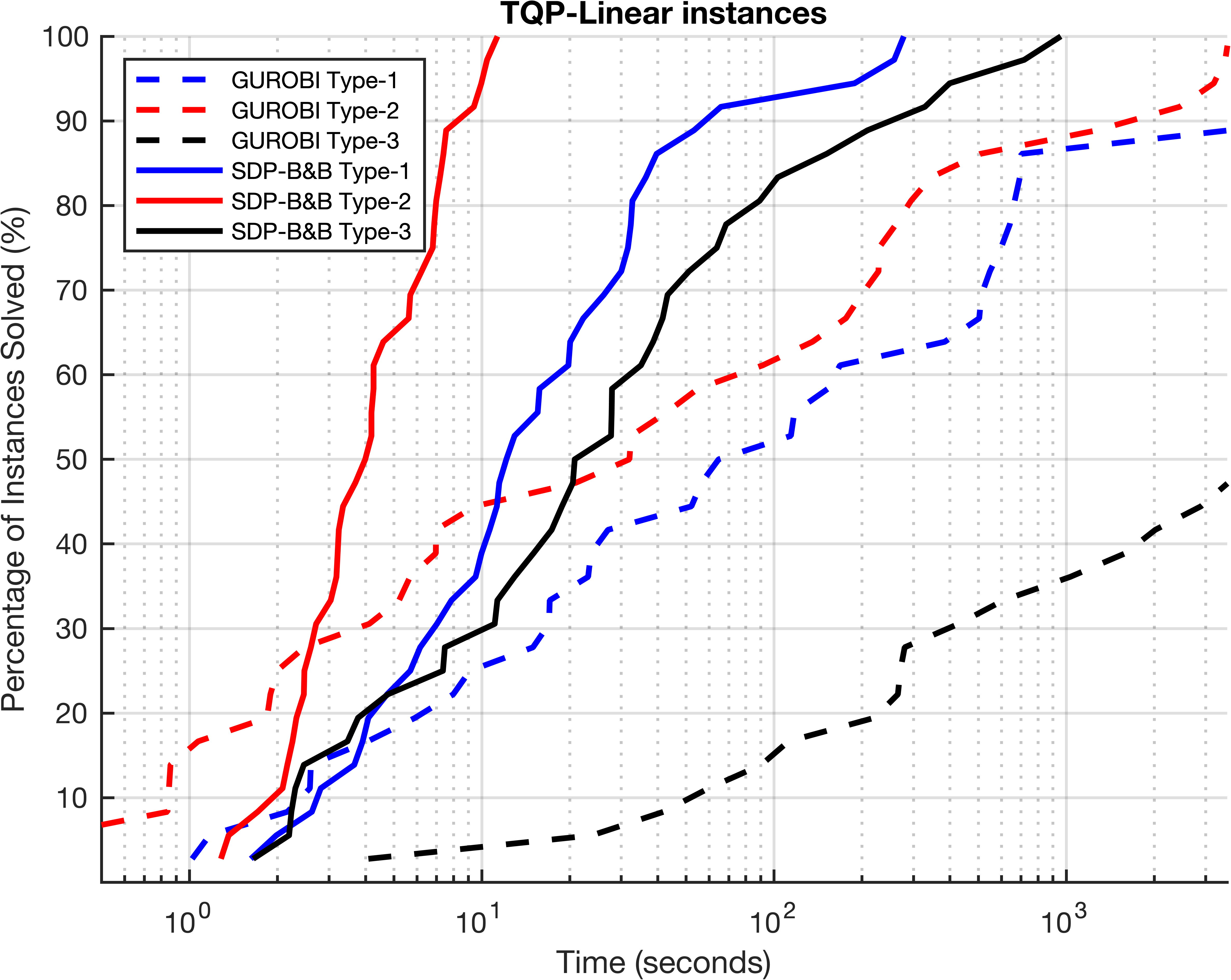}
    \includegraphics[scale=0.21]{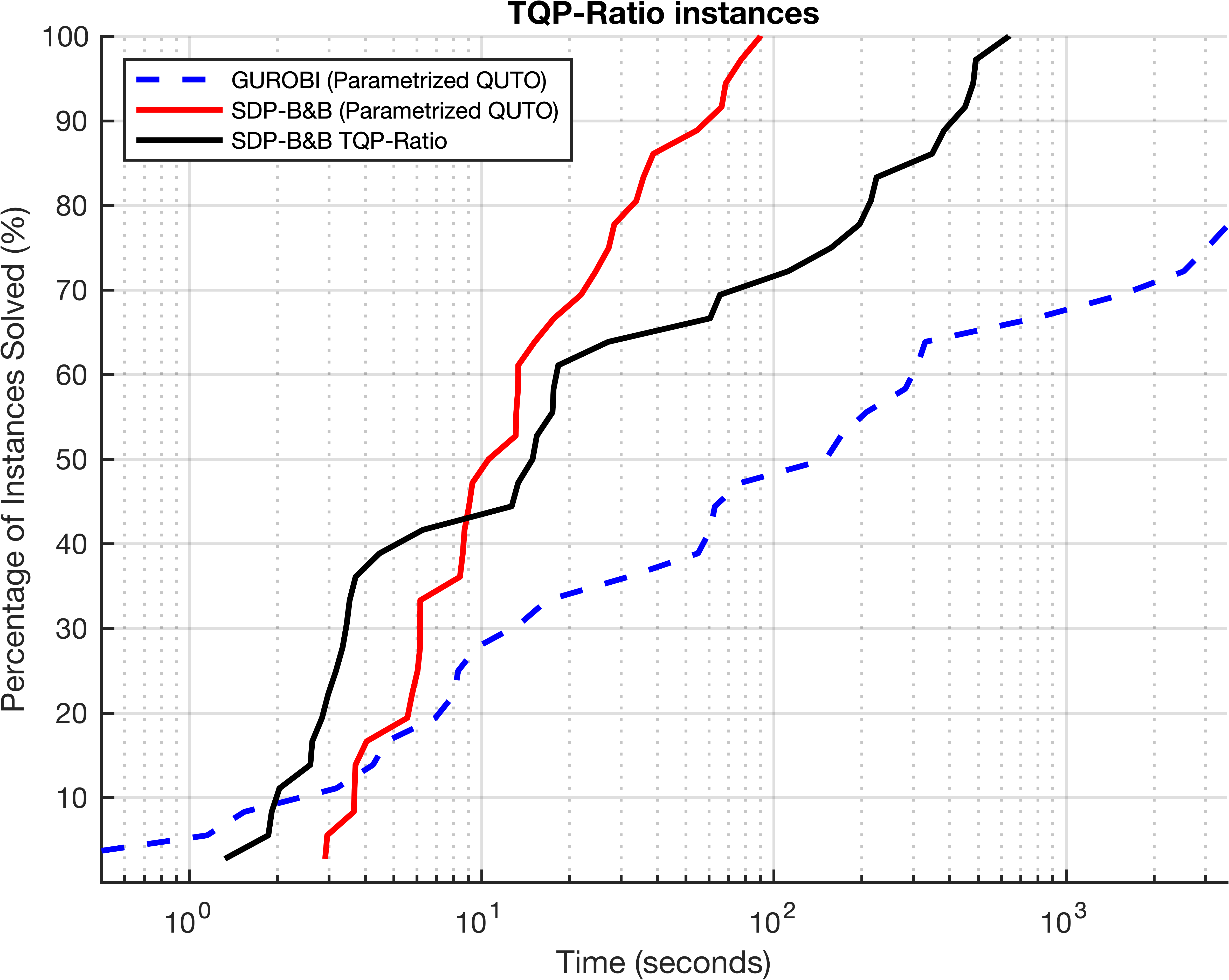}
    \caption{Survival plots comparing {\tt SDP-B\&B} with {\tt GUROBI} on QUTO, TQP-Linear and TQP-Ratio instances of size $n \in \{60, 70, 80 ,90\}$. The plot reports the percentage of problems solved over time, where the time is represented on a logarithmic scale.}
    \label{fig:surv}
\end{figure}

\begin{table*}[!ht]
\centering
\tiny
\caption{Average results of {\tt SDP-B\&B} for the QUTO, TQP-Linear, and TQP-Ratio instances ($n=100\text{--}120$). All metrics are averaged over three random seeds. Column ``sol'' reports the number of instances successfully solved. \textcolor{black}{Whenever sol $ < 3$, we report in brackets the maximum gap over the three seeds.}}
\label{tab:avg_results_large}
\begin{subtable}{0.48\textwidth}
\centering
\caption{QUTO}
\begin{tabular}{rrlrrr}
\toprule
$n$ & $p$ &  \% gap \textcolor{black}{($\max$)} & nodes & time [s] & sol\\
\midrule
\multicolumn{6}{c}{{Type-1}} \\
\midrule
100 & 25 & 0.00 & 6.00 & 101.40 & 3 \\
100 & 50 & 0.00 & 4.00 & 96.27 & 3 \\
100 & 75 & 0.00 & 33.00 & 366.03 & 3 \\
110 & 25 & 0.00 & 20.00 & 541.19 & 3 \\
110 & 50 & 0.00 & 114.00 & 3752.82 & 3 \\
110 & 75 & 0.00 & 112.00 & 1899.71 & 3 \\
120 & 25 & 0.00 & 79.00 & 5026.98 & 3 \\
120 & 50 & 0.94 \textcolor{black}{(2.21)} & 53.00 & 7807.28 & \textcolor{black}{1} \\
120 & 75 & 0.00 & 184.00 & 8342.39 & 3 \\
\midrule
\multicolumn{6}{c}{{Type-2}} \\
\midrule
100 & 25 & 0.00 & 1.00 & 8.81 & 3 \\
100 & 50 & 0.00 & 1.00 & 15.86  & 3\\
100 & 75 & 0.00 & 4.00 & 49.50  & 3\\
110 & 25 & 0.00 & 1.00 & 17.70  & 3\\
110 & 50 & 0.00 & 1.00 & 16.17  & 3\\
110 & 75 & 0.00 & 6.00 & 163.61  & 3\\
120 & 25 & 0.00 & 1.00 & 23.02  & 3\\
120 & 50 & 0.00 & 2.00 & 66.19  & 3\\
120 & 75 & 0.00 & 8.00 & 217.43  & 3\\
\midrule
\multicolumn{6}{c}{{Type-3}} \\
\midrule
100 & 25 & 0.00 & 57.00 & 1338.23 & 3 \\
100 & 50 & 0.00 & 80.00 & 1893.25  & 3\\
100 & 75 & 0.62 \textcolor{black}{(1.85)} & 101.33 & 4554.90 & 2\\
110 & 25 & 0.00 & 91.00 & 3249.31  & 3\\
110 & 50 & 0.31 \textcolor{black}{(0.94)} & 96.33 & 5544.80  & 2\\
110 & 75 & 2.05 \textcolor{black}{(3.10)} & 160.33 & 10567.30 & 1\\
120 & 25 & 0.96 \textcolor{black}{(2.45)} & 118.33 & 8424.84 & 1\\
120 & 50 & 0.54 \textcolor{black}{(1.63)} & 141.33 & 8283.22 & 2\\
120 & 75 & 1.29 \textcolor{black}{(1.95)} & 151.33 & 10782.52 & 1\\
\bottomrule
\end{tabular}
\end{subtable}
\hfill
\begin{subtable}{0.48\textwidth}
\centering
\caption{TQP-Linear}
\begin{tabular}{rrlrrr}
\toprule
$n$ & $p$ & \% gap \textcolor{black}{($\max$)} & nodes & time [s]  & sol\\
\midrule
\multicolumn{6}{c}{{Type-1}} \\
\midrule
100 & 25 & 0.00 & 8.00 & 154.34 & 3 \\
100 & 50 & 0.00 & 7.00 & 172.36 & 3\\
100 & 75 & 0.00 & 4.00 & 135.71 & 3\\
110 & 25 & 0.00 & 16.00 & 369.44 & 3\\
110 & 50 & 0.00 & 81.00 & 2811.27 & 3\\
110 & 75 & 0.00 & 24.00 & 1146.84 & 3\\
120 & 25 & 0.00 & 47.00 & 2097.98 & 3\\
120 & 50 & 0.61 \textcolor{black}{(1.49)} & 105.33 & 8044.35 & 1\\
120 & 75 & 0.00 & 67.00 & 4059.90 & 3\\
\midrule
\multicolumn{6}{c}{{Type-2}} \\
\midrule
100 & 25 & 0.00 & 1.00 & 7.34 & 3 \\
100 & 50 & 0.00 & 1.00 & 12.50  & 3\\
100 & 75 & 0.00 & 1.00 & 17.83  & 3\\
110 & 25 & 0.00 & 1.00 & 15.33  & 3\\
110 & 50 & 0.00 & 1.00 & 12.90  & 3\\
110 & 75 & 0.00 & 1.00 & 24.96  & 3\\
120 & 25 & 0.00 & 7.00 & 87.87  & 3\\
120 & 50 & 0.00 & 1.00 & 36.95  & 3\\
120 & 75 & 0.00 & 2.00 & 145.60  & 3\\
\midrule
\multicolumn{6}{c}{{Type-3}} \\
\midrule
100 & 25 & 0.00 & 175.00 & 3747.39  & 3\\
100 & 50 & 0.44 \textcolor{black}{(1.31)} & 101.33 & 3808.51 & 2\\
100 & 75 & 0.45 \textcolor{black}{(1.35)} & 121.33 & 4455.57 & 2\\
110 & 25 & 0.00 & 93.00 & 3280.33 & 3\\
110 & 50 & 0.20 \textcolor{black}{(0.60)} & 115.67 & 4904.71 & 2\\
110 & 75 & 1.59 \textcolor{black}{(2.64)} & 189.00 & 9918.83 & 1\\
120 & 25 & 1.35 \textcolor{black}{(2.47)} & 81.33 & 7816.62 &1\\
120 & 50 & 0.39 \textcolor{black}{(1.17)} & 114.67 & 7203.70 & 2 \\
120 & 75 & 1.29 \textcolor{black}{(2.15)} & 133.33 & 9628.57 &1\\
\bottomrule
\end{tabular}
\vspace{0.25cm}
\hfill
\end{subtable}
\begin{subtable}{0.48\textwidth}
\centering
\caption{TQP-Ratio (Parametrized QUTO)}
\begin{tabular}{rrrrrr}
\toprule
$n$ & $d$ & iter & nodes & time [s] & sol \\
\midrule
100 & 25 & 1.00 & 10.00 & 335.98 & 3\\
100 & 50 & 1.00 & 138.00 & 4640.18 & 3\\
100 & 75 & 1.00 & 92.00 & 2632.86 & 3\\
110 & 25 & 1.00 & 45.00 & 2064.74 & 3\\
110 & 50 & 1.00 & 83.00 & 5179.87 & 2\\
110 & 75 & 1.00 & 104.67 & 7915.80 & 1\\
120 & 25 & 1.00 & 68.00 & 3715.54 & 3\\
120 & 50 & 1.00 & 93.33 & 10800.00 & 0\\
120 & 75 & 1.00 & 82.33 & 9571.01 & 1\\
\bottomrule
\end{tabular}
\end{subtable}
\hfill
\begin{subtable}{0.48\textwidth}
\centering
\caption{TQP-Ratio}
\begin{tabular}{rrlrrr}
\toprule
$n$ & $d$ & \% gap \textcolor{black}{($\max$)} & nodes & time [s] & sol\\
\midrule
100 & 25 & 0.00 & 92.00 & 1090.39 & 3\\
100 & 50 & 0.25 \textcolor{black}{(0.62)} & 441.33 & 8362.24 & 1\\
100 & 75 & 0.23 \textcolor{black}{(0.59)} & 580.67 & 8525.15 & 1\\
110 & 25 & 0.25 \textcolor{black}{(0.75)} & 147.67 & 5192.44 & 2\\
110 & 50 & 0.53 \textcolor{black}{(1.08)} & 393.67 & 10800.00 & 0\\
110 & 75 & 0.60 \textcolor{black}{(0.87)} & 631.00 & 10800.00 & 0\\
120 & 25 & 0.29 \textcolor{black}{(0.86)} & 168.33 & 5917.93 & 2\\
120 & 50 & 1.37 \textcolor{black}{(1.73)} & 370.33 & 10800.00 & 0\\
120 & 75 & 0.14 \textcolor{black}{(0.41)} & 309.67 & 6206.39 & 2\\
\bottomrule
\end{tabular}
\end{subtable}
\end{table*}

In Figure \ref{fig:surv}, we summarize the comparison between {\tt SDP-B\&B} and {\tt GUROBI} on the three classes of problems. The plot reports the percentage of instances solved by each method over time, with time displayed on a logarithmic scale on the x-axis.

In the upper-left panel, results for the QUTO instances are shown. {\tt SDP-B\&B} clearly outperforms {\tt GUROBI}: it solves all instances within the one-hour time limit, whereas {\tt GUROBI} cannot. Indeed, {\tt GUROBI} fails to solve 6 out of 36 instances of Type-1 of sizes 80 and 90, leaving residual gaps between 1.8\% and 6\%. For Type-2 instances, {\tt GUROBI} fails on 4 out of 36 instances of size 90, with residual gaps ranging from 2.2\% to 5.5\%. Type-3 instances are the most challenging: {\tt GUROBI} cannot solve 19 instances with sizes 70, 80, and 90, with residual gaps between 4.5\% and 25\%.

The upper-right panel reports results for the TQP-Linear instances. Again, {\tt SDP-B\&B} solves all instances within the time limit, whereas {\tt GUROBI} does not. For Type-1 instances, {\tt GUROBI} fails on 5 instances of sizes 80 and 90, with residual gaps between 0.02\% and 0.07\%. For Type-2 instances, only one instance remains unsolved, with a residual gap of 1.4\%. For Type-3 instances, {\tt GUROBI} fails on 19 instances of sizes between 70 and 90, with residual gaps ranging from 0.03\% to 0.1\%.

In the second row of Figure \ref{fig:surv}, we compare three approaches: (i) the Dinkelbach algorithm where the QUTO subproblem is solved by {\tt SDP-B\&B} (Parametrized QUTO); (ii) the Dinkelbach algorithm where the QUTO subproblem is solved by {\tt GUROBI}; and (iii) the direct solution of the TQP-Ratio via {\tt SDP-B\&B}, exploiting relaxation~\eqref{eq:sdp_relax}. The Dinkelbach algorithm is initialized with the solution obtained by the VNS-Ratio heuristic, which turned out to be optimal on all tested instances; hence, only one QUTO subproblem needs to be solved for each instance. 
The plot again shows the percentage of instances solved over time. Here, {\tt GUROBI} is clearly outperformed by both of our approaches, which solve all instances within the time limit, whereas {\tt GUROBI} fails on 18 out of 36 instances. Moreover, our implementation of the parametrized QUTO is generally faster than solving the TQP-Ratio  directly.

Table~\ref{tab:avg_results_large} reports the average performance of {\tt SDP-B\&B} on large-scale instances with $n \in \{100, 110,120\}$ for the QUTO, TQP-Linear, and TQP-Ratio. \textcolor{black}{Results are averaged over the three seeds per parameter configuration. Whenever there are unsolved instances, we report both the average and the maximum gap on the three instances, with the maximum in brackets.} Overall, the method scales well up to \(n=120\), with performance strongly influenced by the instance type.

For both the QUTO and TQP-Linear, Type-2 instances are the easiest, typically solved at the root node or with very few nodes due to the favorable low-rank structure, which leads to tight relaxations. Type-1 instances show moderate growth in difficulty with \(n\), but most cases are still solved to optimality, with only small residual gaps at the largest size (at most \(\textcolor{black}{2.21}\%\) for the QUTO and \(\textcolor{black}{1.49}\%\) for the TQP-Linear). Type-3 instances are the most challenging, as the structure weakens the bounds and results in larger search trees and longer running times. Nevertheless, even in these cases {\tt SDP-B\&B} produces high-quality solutions, with remaining gaps typically below about \(\textcolor{black}{3}\%\) (maximum \(\textcolor{black}{3.1}\%\) for the QUTO and \(\textcolor{black}{2.64}\%\) for the TQP-Linear).

The TQP-Ratio results highlight some differences between the parametrized QUTO (subtable (c))  and the direct solution (subtable (d)) approaches. The parametric method is often more efficient, but when the iterative procedure does not terminate within the time limit, it cannot provide a valid optimality gap for the original ratio problem. In contrast, the direct solution approach always produces a lower bound and therefore a certified gap, even when prematurely terminated. The reported gaps for the direct solution approach remain small overall, ranging between \(0.14\%\) and \(1.37\%\) on the unsolved instances. Moreover, there are configurations in which the direct approach outperforms the parametric one, as for example for $n=120$ and $d=\{50,75\}$. 
Hence, the direct solution approach remains valuable whenever solution certification is required, or when the Dinkelbach scheme is not initialized with a high-quality solution, in which case multiple outer iterations may be needed. In our experiments, the VNS heuristic consistently identifies a solution that is already optimal, so that no further Dinkelbach updates are necessary. Consequently, the computational effort of the parametric method essentially corresponds to solving a single QUTO instance, explaining its good practical performance.

\subsection{\textcolor{black}{Branch-and-bound-results on graphs}}
\textcolor{black}{In this section, we further evaluate the proposed framework on structured instances arising from graph optimization problems. We consider two classes of problems: a weighted ternary
max-cut problem, which is a special case of QUTO, and a ternary quadratic ratio problem motivated by~\cite{Amaral2015CopositivitybasedAF}.}

\textcolor{black}{
\paragraph{Weighted ternary max-cut problem}
Let $G=(V,E)$ be an undirected {edge}-weighted graph with $n$ {vertices} and let $L_w$ denote its weighted Laplacian. We consider $\max_{x\in\{0,\pm1\}^n} x^\top L_w x$ where edge weights may be negative \cite{barahona1988application}. 
This optimization problem is closely related to the Blume–Capel model, which extends the standard Ising model by allowing spins to take ternary values \cite{Capel1966}.
The values $x_i\in\{-1,1\}$ assign vertex $i$ to one of the two sides of the cut, 
whereas $x_i=0$ 
represents a neutral state.
The resulting problem is a QUTO with an objective matrix that is, in general, indefinite.
We consider 18 instances with $n\in\{100,125,150\}$, including both dense 
Ising  instances arising from Ising spin-glass models in statistical physics
and sparse toroidal-grid graphs, see~\cite{Liers2004Ising}.}

\textcolor{black}{
\paragraph{Ternary maximum-density subgraph problem}
We consider the ternary
quadratic ratio problem $ \max_{x\in\{0,\pm1\}^n}
\frac{x^\top A x}{x^\top x+1}$, where $A$ is the adjacency matrix of an undirected graph $G = (V,E)$, see~\cite{Amaral2015CopositivitybasedAF, Trevisan2009MaxCut}. The values $x_i\in\{-1,1\}$ represent two decided classes, while $x_i=0$ represents an undecided vertex. The numerator rewards edges between vertices assigned to the
same class and penalizes edges between vertices assigned to opposite classes, while the denominator normalizes the score by the number of decided vertices.
We consider 25 benchmark graphs with sizes ranging from $n=70$ to $n=162$. The \texttt{can} and \texttt{dwt} instances are derived from structural
matrices of the Harwell--Boeing collection \cite{DuffGrimesLewis1989}; the \texttt{queen} instances are DIMACS graph-coloring benchmarks
\cite{JohnsonTrick1996}; the \texttt{mesh} instances are sparse mesh graphs used in graph partitioning benchmarks \cite{FerreiraEtAl1998}; 
{ and the instances of near-regular graphs \texttt{reg\_n} are from \cite{GaarRendl2020}. Remaining instances we tested our algorithms on are from \cite{Gaar2024}. The latter two classes were used to test algorithms for the stable set problem.}
All 25 instances are solved using the Dinkelbach scheme where at each iteration, the resulting parametric QUTO subproblem is solved either by {\tt GUROBI} or by {\tt SDP-B\&B}.
}
\newline\newline
\textcolor{black}{We set the time limit to 3600 seconds and use the same setting described in Section~\ref{sec:impl}. 
Table~\ref{tab:QUTOgraph} shows a substantial difference in robustness between
the two methods on the weighted ternary max-cut instances. {\tt SDP-B\&B}
solves all 18 instances, whereas
{\tt GUROBI} reaches the time limit of 3600 seconds on 12 instances. On those instances, the residual {\tt GUROBI} gaps range from $0.85\%$ to
$13.23\%$. On the other hand, {\tt SDP-B\&B} yields very small search trees: at most 10 B\&B nodes are processed on every instance, and
six instances are solved directly at the root node after the cutting-plane
procedure. For all instances with
$n\leq125$, the running time of {\tt SDP-B\&B} is below 360 seconds. The
$n=150$ instances require more computational time, but all 6 are still solved in at most 1044 seconds. Observe that {\tt SDP-B\&B} outperforms {\tt GUROBI} on both sparse and dense weighted ternary max-cut instances.}

\begin{table*}[htbp]
\centering
\footnotesize
\caption{\textcolor{black}{Comparison between {\tt GUROBI} and {\tt SDP-B\&B} on the
weighted ternary max-cut instances. ``Nodes'' denotes the number of nodes processed by {\tt SDP-B\&B}}.}\label{tab:QUTOgraph}
\label{tab:gurobi_sdp_results} \color{black}
\begin{tabular}{lrrrrrrr}
\toprule
& & &
\multicolumn{2}{c}{{\tt GUROBI}} &
\multicolumn{3}{c}{{\tt SDP-B\&B}} \\
\cmidrule(lr){4-5}
\cmidrule(lr){6-8}
Graph & $n$ & Density &
Gap (\%) & Time (s) &
Gap (\%) & Nodes & Time (s) \\
\midrule
ising2.5-100\_5555 & 100 & 0.99 & 0.85  & 3600 & 0.00 & 4  & 65.80 \\
ising2.5-100\_6666 & 100 & 0.99 & 2.29  & 3600 & 0.00 & 1  & 31.73 \\
ising2.5-100\_7777 & 100 & 0.98 & 0.01  & 359.37   & 0.00 & 10 & 112.11 \\
ising3.0-100\_5555 & 100 & 0.92 & 0.00  & 1.63     & 0.00 & 4  & 95.40 \\
ising3.0-100\_6666 & 100 & 0.92 & 0.00  & 101.30   & 0.00 & 1  & 14.26 \\
ising3.0-100\_7777 & 100 & 0.92 & 0.00  & 14.14    & 0.00 & 4  & 53.80 \\
\midrule
ising2.5-150\_5555 & 150 & 0.96 & 5.99 & 3600 & 0.00 & 1 & 826.28 \\
ising2.5-150\_6666 & 150 & 0.96 & 7.45 & 3600 & 0.00 & 4 & 633.75 \\
ising2.5-150\_7777 & 150 & 0.96 & 0.01 & 1418.89 & 0.00 & 4 & 847.64 \\
ising3.0-150\_5555 & 150 & 0.75 & 2.91 & 3600 & 0.00 & 4 & 854.71 \\
ising3.0-150\_6666 & 150 & 0.75 & 5.02 & 3600 & 0.00 & 4 & 931.17 \\
ising3.0-150\_7777 & 150 & 0.75 & 3.78 & 3600 & 0.00 & 4 & 1043.88 \\
\midrule
t2g10\_5555 & 100 & 0.04 & 4.13  & 3600 & 0.00 & 1  & 28.82 \\
t2g10\_6666 & 100 & 0.04 & 4.13  & 3600 & 0.00 & 1  & 47.69 \\
t2g10\_7777 & 100 & 0.04 & 0.01  & 295.68   & 0.00 & 1  & 26.77 \\
\midrule
t3g5\_5555  & 125 & 0.05 & 7.39  & 3600 & 0.00 & 4  & 176.45 \\
t3g5\_6666  & 125 & 0.05 & 12.41 & 3600 & 0.00 & 4  & 218.69 \\
t3g5\_7777  & 125 & 0.05 & 13.23 & 3600 & 0.00 & 10 & 358.97 \\
\bottomrule
\end{tabular}%
\end{table*}

\textcolor{black}{
The results for the ternary maximum-density subgraph problem are reported in
Table~\ref{tab:RATIOgraph}. A first observation is that every run requires
only one Dinkelbach iteration. Hence, for these instances, the computational
cost of the fractional scheme is essentially the cost of solving a single
parametric QUTO subproblem, consistently with the effectiveness of the
VNS-based initialization observed in the previous experiments.
We can see that {\tt GUROBI} hits the time limit on 3 instances out of 25, whereas {\tt SDP-B\&B} solves all the instances within 2 minutes, apart from dtw\_162, which is solved in about 400 seconds, while {\tt GUROBI} hits the time limit. Unlike the weighted max-cut experiments, however, the performance difference
is dependent on the graph family rather than solely on the 
{size of the graph}.
{\tt GUROBI} is slightly faster on a number of the small and medium-size
instances, and in particular on the \texttt{reg}, \texttt{queen}, and
\texttt{spin5} families. For example, on \texttt{G\_150\_025},
{\tt GUROBI} requires 33.68 seconds compared with 114.77 seconds for
{\tt SDP-B\&B}. On the other hand, {\tt SDP-B\&B} exhibits much better
scalability on the larger sparse \texttt{mesh}, \texttt{dwt}, and \texttt{can}
instances. For \texttt{mesh.138.2.232}, the running time decreases from
2946.68 to 81.31 seconds compared to {\tt GUROBI}, while for \texttt{can\_144} and
\texttt{can\_161} it decreases from 459.95 to 66.21 seconds and from
1466.50 to 101.41 seconds, respectively. The two large
\texttt{mesh} instances and \texttt{dwt\_162} further illustrate this
difference, since {\tt GUROBI} reaches the time limit whereas
{\tt SDP-B\&B} solves them to optimality.}

Overall, our numerical results demonstrate that ISDP is a promising approach for solving ternary quadratic programs.

\begin{table*}[htbp]
\centering
\footnotesize
\caption{\textcolor{black}{Comparison of the Dinkelbach approach using {\tt GUROBI} or
{\tt SDP-B\&B} to solve the parametric QUTO subproblems for the graph ternary
ratio instances. ``Iter'' denotes the number of Dinkelbach iterations.}}\label{tab:RATIOgraph}
\label{tab:ratio_results} \color{black}
\begin{tabular}{lrrrrrr}
\toprule
& & &
\multicolumn{2}{c}{{\tt GUROBI}} &
\multicolumn{2}{c}{{\tt SDP-B\&B}} \\
\cmidrule(lr){4-5}
\cmidrule(lr){6-7}
Graph & $n$ & Density &
Iter & Time (s) &
Iter & Time (s) \\
\midrule
mesh.70.2.120  & 70  & 0.05 & 1 & 6.65     & 1 & 6.79 \\
mesh.74.2.129  & 74  & 0.05 & 1 & 9.26     & 1 & 7.11 \\
mesh.137.2.231 & 137 & 0.02 & 1 & 3600 & 1 & 113.58 \\
mesh.138.2.232 & 138 & 0.02 & 1 & 2946.68 & 1 & 81.31 \\
mesh.148.2.265 & 148 & 0.02 & 1 & 3600 & 1 & 113.25 \\
\midrule
HoG\_34276    & 72  & 0.06 & 1 & 4.84  & 1 & 5.26 \\
G\_80\_050    & 80  & 0.50 & 1 & 5.30  & 1 & 9.06 \\
G\_100\_025   & 100 & 0.22 & 1 & 10.06 & 1 & 18.53 \\
spin5         & 125 & 0.05& 1 & 19.69 & 1 & 22.62 \\
G\_150\_025   & 150 & 0.25 & 1 & 33.68 & 1 & 114.77 \\
\midrule
reg\_n100\_r10 & 100 & 0.10 & 1 & 10.02 & 1 & 13.49 \\
reg\_n100\_r4  & 100 & 0.04 & 1 & 14.16 & 1 & 17.26 \\
reg\_n100\_r6  & 100 & 0.06 & 1 & 10.75 & 1 & 14.04 \\
reg\_n100\_r8  & 100 & 0.08 & 1 & 10.61 & 1 & 13.33 \\
\midrule
queen9\_9   & 81  & 0.33 & 1 & 5.21  & 1 & 9.00 \\
queen10\_10 & 100 & 0.30 & 1 & 9.56  & 1 & 19.67 \\
queen11\_11 & 121 & 0.27 & 1 & 16.34 & 1 & 39.75 \\
queen12\_12 & 144 & 0.25 & 1 & 27.66 & 1 & 77.00 \\
\midrule
dwt\_72  & 72  & 0.03 & 1 & 4.66     & 1 & 7.73 \\
dwt\_87  & 87  & 0.06 & 1 & 9.84     & 1 & 11.58 \\
dwt\_162 & 162 & 0.05 & 1 & 3600 & 1 & 401.32 \\
\midrule
can\_73 & 73  & 0.06 & 1 & 10.17    & 1 & 6.69 \\
can\_96 & 96  & 0.07 & 1 & 19.85    & 1 & 12.91 \\
can\_144  & 144 & 0.06 & 1 & 459.95   & 1 & 66.21 \\
can\_161  & 161 & 0.05 & 1 & 1466.50 & 1 & 101.41 \\
\bottomrule
\end{tabular}%
\end{table*}

\bibliography{myRefs.bib}

\appendix

\section{Instance generation}
\label{app:instance_generation}
We use the following three types of data generators to obtain instance for the TQP-Linear~\eqref{turbine Problem}.
\begin{enumerate}
\item[(i)] Type-1 instances are generated as follows, see also \cite{Buchheim2012SemidefiniteRF}.
For given $n \in \mathbb{N}$ and $p\in \{25,50,75\}$, we sample $\lfloor pn/100 \rfloor$ values uniformly from the interval $[-1,0]$,
and the remaining  $n-\lfloor pn/100 \rfloor$ values uniformly from $[0,1]$. These values form the vector $\mu$.
Then, we generate an $n\times n$ matrix whose elements are drawn uniformly at random from $[-1,1]$
and orthonormalize its columns to obtain vectors $v_i$. Finally, we set $Q= \sum_{i=1}^n \mu_i v_i v_i^\top.$ 

\item[(ii)] Type-2 instances provide matrices $Q$ whose rank is $\lceil n/2\rceil$, see also \cite{Buchheim2018SDPbasedBF} for a similar generator.
For given $n \in \mathbb{N}$, we generate an $n\times n$ matrix whose elements are drawn uniformly at random from $[-1,1]$
and orthonormalize its columns to obtain vectors $v_i$.
Then, the first $\lfloor n/2\rfloor $ entries of  $\mu$ are set to zero, and each remaining entry is nonnegative with probability $p/100$  and zero otherwise.
We consider $p\in \{25,50,75\}$. 

\item[(iii)]  Type-3 instances provide random sparse matrices $Q$, see also \cite{Buchheim2018SDPbasedBF}.
For given $n \in \mathbb{N}$ and $p\in \{25,50,75\}$, each entry of the matrix $Q$ is zero with probability $1-p/100$ and the remaining entries are drawn uniformly at random from 
the interval $[-1,1]$.
\end{enumerate}
In all above types of instance generators, elements of $c$ are drawn uniformly at random from $[-1,1]$.

We generate QUTO  instances that do not reduce to max-cut instances. In particular, we use the same three types of generators as described above and take the absolute values of the diagonal elements, ensuring the existence of positive diagonal elements, see \Cref{corr:maxcut}.

We generate  instances  for the TQP-Ratio~\eqref{eq:frac_ternary} as follows. For each test instance, the numerator coefficients $(A,a,a_0)$ and denominator coefficients $(B,b,b_0)$
are generated independently at random. The quadratic matrices $A,B\in\mathcal{S}^{n}$ are  sparse: given a target density $d \in [0, 100]$, each entry in the lower triangular part is set to a nonzero value with probability
$d/100$ and to zero otherwise, and symmetry is enforced by mirroring to the upper triangular
part. Whenever an entry is nonzero, it is sampled uniformly from the integer interval $\{-50,\ldots,50\}$.
The linear vectors $a,b\in\mathbb{R}^n$ and the constant term $a_0$ are generated independently with
i.i.d.\ entries uniformly distributed over $\{-50,\ldots,50\}$. 
We set $b_0 \;=\; 1 + \sum_{i=1}^n |B_{ii}| \;+\; 2\sum_{1\le i<j\le n} |B_{ij}| \;+\; \sum_{i=1}^n |b_i|,$
which ensures $g(x)\ge 1$ for all  $x\in\{0,\pm 1\}^n$.

\end{document}